\documentclass[preprint,10pt]{elsarticle}
\usepackage{natbib}
\usepackage{url}
\usepackage{bm}
\usepackage{amsthm,amsmath,amsfonts,amssymb}
\usepackage{color}
\usepackage{tikz}
\usepackage{graphicx}
\usepackage{enumitem}
\usepackage{verbatim}
\newcommand*{\dt}{%
{\mbox{\raisebox{0.07em}{\large\bfseries .}}}}
  
\fontsize{10pt}{10pt}\selectfont
\fontsize{9pt}{9pt}\selectfont

\newtheorem{Lemma}{Lemma}
\newtheorem{Corollary}{Corollary}
\newtheorem{Proposition}{Proposition}

\theoremstyle{remark}
\newtheorem{Example}{Example}
\newtheorem{Remark}{Remark}

\providecommand{\keywords}[1]
{
  \small	
  \textbf{\textit{Keywords---}} #1
}

\def\pgf{\emph{p.g.f.~}}
\def\pgfl{\emph{p.g.fl.~}}
\def\iid{\emph{i.i.d.}~}

\def\a{{\mathfrak a}}
\def\B{{\mathfrak B}}
\def\b{{\mathfrak b}}
\def\g{{\mathfrak g}}

\def\w{{\mathfrak w}}
\def\m{{\mathfrak m}}
\def\n{{\mathfrak n}}
\def\p{{\mathfrak p}}
\def\v{{\mathfrak v}}

\def\z{{\mathfrak z}}
\def\s{{\mathfrak s}}
\def\r{{\mathfrak r}}
\def\M{{\mathfrak M}}
\def\Q{{\mathfrak Q}}
\def\V{{\mathfrak V}}
\def\Z{{\mathfrak Z}}
\def\G{{\mathfrak G}}
\def\d{{\textcolor{black}{q}} }
\def\N{{\textcolor{black}{d}}}
\usepackage[para]{footmisc} 

\newpageafter{abstract}
\begin{document}
\begin{frontmatter}
\title{
{A Class of Gaussian Fields on $\mathbb{Z}_q^d$}\\
}
 \author{Robert Griffiths}
\affiliation{
organization={School of Mathematics},
            addressline={Monash University},
            country={Australia}
         }
	
\ead{Bob.Griffiths@Monash.edu}
 \author{Shuhei Mano}
\affiliation{
organization={The Institute of Statistical Mathematics},
            addressline={Tokyo},
            country={Japan}
         }	
\ead{smano@ism.ac.jp}

\date{}

\begin{abstract}
Complex Gaussian fields $(g_x)$ on $\mathbb{Z}_q^d$ are constructed from 
a class of reversible long range random walks $(X_t)_{t\in \mathbb{N}}$ on $\mathbb{Z}_\d^\N$ in \citeauthor{GM2025a} (\citeyear{GM2025a}). The construction is from taking 
$\mathbb{E}\big [g_x\bar{g}_y\big ]= (1-\alpha)G(x,y;\alpha)$, where $G(x,y;\alpha)$ is the Green function of a random walk with killing in each transition at rate $1-\alpha$. Another real Gaussian field $(h^\circ_x=\Re g_x + \Im g_x)$ is constructed such that also Cov$(h^\circ_x,h^\circ_y) = (1-\alpha)G(x,y;\alpha)$.
If the random walk is not reversible, then a different approach is needed because $G(x,y;\alpha)$ 
is not exchangeable in $x,y$. It is shown that there exists two real Gaussian fields $(h^+_x)$, $(h^-_y)$ such that Cov$(h^+_x,h^-_y) = (1-\alpha)G(x,y;\alpha)$. A decomposition of the Gaussian field $(g_x)$ into a sum of independent Gaussian random variables is made. By letting $q\to \infty$ the Gaussian field becomes defined from an infinite-dimensional random walk on a torus. The random walk model is also extended to $d=\infty$ by considering a de Finetti random walk where entries in the increments of the random walk are exchangeable. A limit Gaussian field on $\mathbb{R}^d$ arises from a central limit theorem approach. The transform of this Gaussian field, which is again a Gaussian field, is calculated. It has a simpler covariance matrix than the original field.
The Hamiltonian connected to the Gaussian field is calculated. A limit theorem for the partition function arising from the Hamiltonian is found.

\medskip

\noindent
\keywords \ Gaussian Fields on $\mathbb{Z}_q^d$, Long range random walks, Green functions of random walks, Multivariate Krawtchouk polynomials.

\end{abstract}

\end{frontmatter}

\section{Introduction}
\citeauthor{GM2025a} (\citeyear{GM2025a}) develop a theory of a class of reversible long range random walks $(X_t)_{t\in \mathbb{N}}$ on $Z_\d^\N$. 
 In a transition
\begin{equation}
X_{t+1} = X_t + Z_t \mod \d,
\label{rweqn:000}
\end{equation}
where $(Z_t)_{t=1}^\infty$ are independent and identically distributed (\iid).
In the $\N$-dimensional random walk 
$(X_t)_{t\in \mathbb{N}}$ each entry performs a random walk on 
${\cal V}_\d = \{0,1,\ldots, \d-1\}$, $\mod \d$. The state space of $(X_t)$ is denoted by ${\cal V}_{\d,\N}= {\cal V}_\d^\N$.
A characterization is made of this class of random walks which have a spectral expansion of the transition function with eigenvectors $(\theta_1^{x\dt r})_{r\in {\cal V}_{\d,\N}}$,
where $\theta_1=e^{2\pi i/q}$. In this characterization the eigenvalues have the form
$\rho_r=\mathbb{E}\big [\theta_1^{V\dt r}\big ]$ where $V\in {\cal V}_{q,d}$ is  a random variable. The resulting random walk has entries which are mixtures of circulant random walks on ${\cal V}_{\d,\N}$, controlled by $V$.
The random walk is extended to a continuous state space on a $d$-dimensional torus in Proposition \ref{torusprop}.
Although $(X_t)$ is a discrete time process a Poisson embedding forms a process in continuous time in Section \ref{ctstime}. 

A class of random walks on a hypercube when $d=2$ with more general independent Bernoulli$(p)$ entries for $Z_t$ is studied in \cite{CG2021}.

We consider the Green function $G(x,y;\alpha)$ where there is a probability of killing in each transition of $1-\alpha$, $\alpha \in (0,1)$ in the processes $(X_t)$. The Green function is then finite because the killed process is transient. 
A spectral expansion of $G(x,y;\alpha)$ is derived in Proposition \ref{proposition:expansions}. Attention is given to the particular case when $V$, and therefore $X_t$, have an exchangeable distribution of entries. Then it is appropriate to consider the counts of entries $n_t$ where $n_t[j]=|\{j:X_t[k]=j, k = 1,\ldots, d\}|$. The stationary distribution of $(n_t)$ is multinomial$\big (d,(1/q)\big )$.
Eigenvectors in the count process $(n_t)$ are multivariate Krawtchouk polynomials with a basis 
$(\theta_1^l)_{l\in {\cal V}_q}$, \citep{G1971,M2004,MT2004,M2011,DG2014}.

It is possible to consider a random walk on ${\cal V}_{q,\infty}$, provided the increments $(Z_t)$ are well defined. A natural choice is to take the entries of $Z_t$ to be from a 
de Finetti sequence. The entries then have an exchangeable distribution. Considering $d$ entries from the de Finetti sequence, the form of
$G(x,y;\alpha)$ is derived. It is shown in Lemma \ref{mainlemma} that the eigenvalues in a spectral expansion of $(1-\alpha)G(x,y;\alpha)$, grouped by indices $l$ have the form of moments $\mathbb{E}\Big [\prod_{k=1}^{\d-1}Y[k]^{l_k}\Big ]$,
where $Y$ is random with $q$ entries derived from products of points in a $q$-dimensional point process.

 A Gaussian free field is defined as a collection of zero mean Gaussian random variables, denoted by $(g_x)_{x \in \cal{V}}$,  for some vertex set $\cal{V}$,
 such that any finite sub-collection of these  random variables is jointly distributed multivariate normal. The covariance structure of the field is defined by the Green function of a  Markov Chain. 
For an introduction to Gaussian fields see \cite{WP2020} and \cite{B2016}.
In this paper a study is made of particular complex Gaussian free fields $(g_x)$, indexed by ${\cal V}_{q,d}$,
with means zero and 
$
\mathbb{E}\big [g_x\bar{g}_y \big] = (1-\alpha)G(x,y;\alpha)
$,
or a variation when elements of $x,y$ are grouped by type. $(g_x)$, being complex, is not uniquely specified by $\mathbb{E}\big [g_x\bar{g}_y \big]$. An assumption in most of the paper is that the eigenvalues of the Green function are real which is equivalent to $G(x,y;\alpha) = {G(y,x;\alpha)}$. In this exchangeable case there is a 
real Gaussian field $(h^\circ_x)$, uniquely specified by its covariance function, with
$
\text{\rm Cov}(h^\circ_x,h^\circ_y) = \mathbb{E}\big [g_x\bar{g}_y \big ] = (1-\alpha)G(x,y;\alpha)
$, defined by 
$
h^\circ_x = \Re g_x + \Im g_x
$.
Proposition \ref{complextoreal:00} details the construction of $(h^\circ_x)$. When $G(x,y;\alpha)$ is not exchangeable in $x,y$ Proposition \ref{complextoreal:01} shows how to construct two real Gaussian fields $(h_x^+), (h_x^-)$ such that the cross covariance 
Cov$(h_x^+,h_y^-) = (1-\alpha)G(x,y;\alpha)$.
A class of Gaussian fields on the hypercube when $d=2$ is studied in \cite{CG2025}. Gaussian free fields on Hamming graphs $H(d,n)$ based on the Hamming distance between vertices are studied in \cite{M2025}.
Decomposition expressions of the Gaussian field into a sum of independent Gaussian variables are made. In a model where the entries of $X_t$ are exchangeable a central limit theorem is found for the Green function and a limit Gaussian field on $\mathbb{R}^d$ constructed 
in Proposition \ref{gfclt:00} when $d\to \infty$. The Gaussian field of a transform is calculated in Proposition \ref{gftransform:00}. The covariance structure of the transform is simpler than in the actual field.

In Section \ref{Hamiltonian:00} we consider the Hamiltonian which arises from the Gaussian field $(h^\circ_x)$, and the underlying Markov chain which it is constructed from. A limit theorem for the partition function 
$Z=\mathbb{E}\big [\sum_{x\in {\cal V}_{q,d}}e^{\beta g_x}\big ]$ is calculated when $d \to \infty$. Additionally a limit result when $\alpha \to 1, d\to \infty$ is calculated.

It is natural to think of the Potts model because of the huge literature on this model and the fact that our Gaussian Fields are defined on $\mathbb{Z}_q^d$. 
  The analogy of our model with the Potts model is briefly discussed in Section \ref{Pottssection}.

\section{Stochastic processes from circulants}
Circulant matrices play an important role in random walks (\ref{rweqn:000}). 
A circulant transition probability matrix in 1-dimension is 
\begin{equation*}
P(v) =
\begin{bmatrix}
v_0&v_1&\cdots &v_{\d-1}\\
v_{\d-1}&v_0&\cdots&v_{\d-2}\\
\cdots&\cdots&\cdots &\cdots\\
\cdots&\cdots&\cdots &\cdots\\
v_1&v_2&\cdots &v_0
\end{bmatrix}.
\end{equation*}
It is very well known and exploited in \cite{GM2025a} that a spectral expansion of $P^t(v)$ is 
\begin{equation*}
P^{(t)}_{xy} = \frac{1}{\d}\sum_{r=0}^{\d-1}\rho_r^t\theta_r^{x-y}
\end{equation*}
where
$\theta_r = e^{2\pi i/\d \dt r}$, the $r^{\text{th}}$ root in the $q^{\text{th}}$ root of unity, and
$\rho_r = \sum_{k=0}^{\d-1}v_k\theta_r^k$.
In the $\N$-dimensional random walk (\ref{rweqn:000}) each entry in $X_t$ behaves as a circulant random walk with random jump probabilities.
Assume that transitions of entries of $X_t$ to $X_{t+1}$ are made according transition matrices $P(V[1]),\ldots, P(V[\N])$ where $V=(V[1],\ldots V[\N])$. The entries of $V$ are random on ${\cal V}_\d$ and not necessarily independent.
The probability of a transition $X_t=x$ to $X_{t+1}=y$ is then
\begin{equation*}
P_{xy}=\mathbb{E}\big [\prod_{k=1}^\N P_{x[k] y[k]}(V[k])\big ]
\end{equation*}
where expectation is over $V$. 
The eigenvectors of $P_{xy}$ are
\begin{equation*}
\prod_{k=1}^\N\theta_{r[k]}^{x[k]}=\theta_1^{x\dt r},
\end{equation*}
where $r\in {\cal V}_{q,d}$ and a complex conjugate form. The eigenvalues are 
\begin{equation*}
\rho_r =   \mathbb{E}\big [\prod_{k=1}^\N\theta_{r[k]}^{V[k]}\big ]
= \mathbb{E}\big [\theta_1^{V\dt r}\big ].
\end{equation*}
A characterization of these transition matrices $P_{xy}$, $x,y\in {\cal V}_{q,d}$ with these eigenvectors
is the following.
\label{Proposition:4}
\begin{Proposition}{(Proposition 6, (\cite{GM2025a})}\label{pcirc}
Let $x,y \in  {\cal V}_{\d,\N}$. Then
\begin{equation}
P_{xy} = \frac{1}{\d^\N}\sum_{r\in {\cal V}_\d }\rho_r
\theta_1^{(x-y)\dt r}
\label{spectral:77}
\end{equation}
is non-negative and therefore a transition probability matrix if and only if
$ \rho_r  = \mathbb{E}\big [\theta_1^{V\dt r}\big ]$ for a random variable $V$ on ${\cal V}_{\d,\N}$.
The stationary distribution is uniform in ${\cal V}_{\d,\N}$.
\end{Proposition}
\begin{Remark}
It is possible to obtain spectral expansions (\ref{spectral:77}) where $\rho_r=0$ if $r$ has more than $c$ entries which are non-zero by the following construction. Let 
$V$ be such that $c$ entries are chosen uniformly with probability ${d\choose c}^{-1}$ to be from an exchangeable distribution $\mathbb{P}(V[i_1]=v_1,\ldots, V[i_c]= v_c) = p(v_1,\ldots,v_c)$ and the other $d-c$ entries are \iid uniform on $1/q^{d-c}$.
Then
\begin{equation*}
\rho_r={d\choose c}^{-1}\sum_{i_1,\ldots,i_c}\theta_1^{r[i_1]v_1+\cdots + r[i_c]v_c}p(v_1,\ldots, v_c)
\delta_{r[i_c+1],0}\cdots \delta_{r[i_d],0},
\end{equation*}
which is zero unless less than or equal to $d-c$ entries of $r$ are zero.
\begin{Example}
If $c=1$ then
\begin{equation*}
P_{xy} = \frac{1}{q^d}\sum_{k=1}^d\frac{1}{d}
\sum_{r=(0,\ldots,r[k],\ldots ,0)}
\rho_r\theta_1^{(x[k]-y[k])r[k]}.
\end{equation*}
If $c=2$ then
\begin{equation*}
P_{xy} = \frac{1}{q^d}\sum_{j,k=1, j\ne k}^d\frac{2}{d(d-1)}
\sum_{r=(0,\ldots,r[j],0,\ldots r[k],\ldots ,0)}
\rho_r\theta_1^{(x[j]-y[j])r[j]+(x[k]-y[k])r[k]}.
\end{equation*}
\end{Example}
\end{Remark}
An extension in \cite{GM2025a} is to a process  $(\B_t)_{t\in \mathbb{N}}$
 which has a state space $[0,1]^\N$ when $\V$ has a wrapped distribution on a $\mathbb{T}^\N$  torus. Then there is a transition density, which is convergent, expressed when $\N=1$ as
\begin{align*}
f\big (\b\mid \a\big )
&=  \sum_{r={-\infty}}^{\infty}
\mathbb{E}\big [e^{2\pi i\dt \V r}\big ]
e^{2\pi i\ (\a -\b)\dt r}\ \a,\b \in \mathbb{T}.
\end{align*}
\begin{Proposition}{(Proposition 7, \cite{GM2025a})} \label{torusprop} Let $(\B_t)_{t\in \mathbb{N}}$ be a homogeneous in time stochastic process on a $[0,1]^\N$ torus with a uniform stationary distribution and eigenvectors $(e^{2\pi i  B_t\dt r})_{r\in \mathbb{N}^d}$. Then the class of such processes is characterized by having eigenvalues
\begin{equation*}
\rho_r = \mathbb{E}\big [e^{2\pi i\V\dt r}].
\end{equation*}
where $\V$ is a wrapped random variable on a $[0,1]^\N$ torus $\mathbb{T}^\N$.
\end{Proposition}
The Green function when $(Z_t)$ are \iid distributed as $V\in {\cal V}_{\d,d}$ with exchangeable entries can be calculated from knowing the transition density of $(X_t)$ in Proposition \ref{pcirc}. We restrict attention to real eigenvalues.
\begin{Proposition}\label{qexpansion}{(Proposition 8, \cite{GM2025a})}
Suppose the entries of $V$ are exchangeable and the eigenvalues $(\kappa_l)$ are real.
Write $M_t = (M_{tj})_{j=0}^{\d-1}$, where $M_{tj}$ counts the number of entries equal to $j$ in $X_t$.
The $t$-step transition function for $M_0=m$ to $M_{t}=n$ has a spectral expansion
\begin{equation}
\mathbb{P}_t\big (n\mid m)=
\p(n;d)
\Bigg \{ 1 + \sum_{l:0<|l| \leq \N}\kappa_l^th_lQ_l(m;(\theta_k))\overline{Q_l(n;(\theta_k))}\Bigg \},
\label{spectral:001}
\end{equation}
where $\big (Q_l(\cdot ;(\theta_k))\big )$ are multivariate Krawtchouk orthogonal polynomials on the uniform multinomial distribution
\[
\p(n;\N) = {\N\choose n}\frac{1}{\d^\N}.
\]
$Q_l(m ;(\theta_k))$ is the coefficient of $w_1^{l[1]}\cdots w_{\d-1}^{l[\d -1]}$ in the generating function
\begin{equation}
\Q(m;w,\d,\N) =
\prod_{j=0}^{\d-1}\Big (1 + \sum_{k=1}^{\d-1}w_k\theta_k^j\Big )^{m[j]}.
\label{genfn:00}
\end{equation}
where $\theta_k= e^{2\pi i/\d\ k}$.
The scale constants for the orthogonal functions
\begin{equation*}
h_l^{-1} = \mathbb{E}\big [Q_l(M;(\theta_k))\overline{Q_l(M;(\theta_k))}\big ]
= \frac{\N!}{(\N-|l|)!l[1]!\cdots l[\d-1]!},
\end{equation*}
where expectation is taken over $M$ having a uniform multinomial distribution
and $|l| = l[1]+\cdots +l[\d-1]$.  
The eigenvalues 
$\kappa_l,\ j \in {\cal V}_\d$ are a grouping of indices in $\rho_r$.
Let  $A_l = \{k: |\{k:r[k]=j\}| = l[j], j \in {\cal V}_\d\}$. Then 
\begin{align*}
\kappa_l =   \mathbb{E}\big [\prod_{A_l} \theta_{r[k]}^{V[k]}\big ]
=\mathbb{E}\big [
\theta_1^{ S_{l[1]} }
\theta_2^{S_{l[2]} }
\cdots
\theta_{\d-1}^{ S_{l[\d-1]} }\Big ]
\end{align*}
where
\begin{align*}
S_{l[1]} &= V[1]+\cdots + V[l[1]]\\
S_{l[2]} &= V[l[1]+1]+\cdots + V[l[1]+1+l[2]]\\
&\cdots\cdots\\
S_{l[\d-1]}&=V[l[1]+\cdots + l[\d-1] +\d-1]] + \cdots +V[l[1]+\cdots + l[\d-2] +\d-1 + l[\d-1]].
\end{align*}
\end{Proposition}
There is a general way of writing the eigenvalues $\kappa_l$, expressed in the following Corollary.
\begin{Corollary}\label{Generaltrans}{(Proposition 13, \cite{GM2025a})}
An expansion (\ref{spectral:001}) with given eigenfunctions 
$\big (Q_l(m;(\theta_k)\big )$ is non-negative if and only if
\begin{equation}
\kappa_l=h_l\mathbb{E}\big [Q_l(V;(\theta_k))\big ].
\label{characterize:00}
\end{equation}
Then, when $t=1$, (\ref{spectral:001}) can be written as 
\begin{equation*}
\p(n;\N)\mathbb{E}\bigg [
 1 + \sum_{|l| \leq \N}h_l^2Q_l(V;(\theta_k))Q_l(m;(\theta_k))\overline{Q_l(n;(\theta_k))}\bigg ],
\label{spectral:0011}
\end{equation*}
with expectation over $V$.\\
\end{Corollary}
Multivariate Krawtchouk polynomials were first introduced by \cite{G1971} and studied in \cite{DG2014}.
The particular form with generating function (\ref{genfn:00}) is important in the study of circulant multitype random walks in 
\cite{M2004,MT2004,M2011,M2018}.
Multivariate Krawtchouk polynomials are complicated, but important, objects which need time to understand.
\subsection{Continuous time processes}\label{ctstime}
A Poisson process embedding of the discrete process to form transition functions in continuous time $\tau \geq 0$ is
for $\lambda >0$,
\begin{equation*}
{\cal P}_\tau(y\mid x;\lambda) = \sum_{k=0}^\infty e^{-\lambda\tau}\frac{(\lambda\tau)^k}{k!}
\mathbb{P}_k(y\mid x ),
\end{equation*}
where $\mathbb{P}_k(y\mid x )=\mathbb{P}\big (X_k=y\mid X_0=x)$.
The eigenvectors of ${\cal P}_\tau(y\mid x;\lambda)$ and $\mathbb{P}_k(y\mid x )$ are the same, however the eigenvalues show the form of a subordinated process. There is a more general form by letting $\lambda \to \infty$ and taking the original eigenvalues tending to 1 to obtain a limit.
General forms in the continuous time processes arising from Propositions \ref{pcirc}, \ref{torusprop}, \ref{qexpansion} are the following.
\begin{itemize}
\item
Proposition \ref{pcirc}.\\
The eigenvalues in ${\cal P}_\tau(y\mid x)$ have the form
\begin{equation*}
\rho_r(\tau) 
= \exp \Big \{\tau \int_{[0,1]^d}
\frac{
\theta_1^{\xi\dt r}-1 
}
{|\xi|}\beta(d\xi)
\Big \},
\end{equation*}
obtained by letting $\lambda\to \infty$ and the measure of $V$ depend on $\lambda$ such that $\lambda$ times the measure tends to $\beta(d\xi)/|\xi|$, $|\xi| > 0$. The rate matrix has elements for $x\ne y$ of
\begin{equation*}
Q_{xy} = \frac{1}{\d^\N}\sum_{r\in {\cal V}_{\d,\N}, |r|\ne 0}\exp \Big \{\tau \int_{[0,1]^d}
\frac{
\theta_1^{\xi\dt r}-1 
}
{|\xi|}\beta(d\xi)\Big \}
\theta_1^{(x-y)\dt r}.
\end{equation*}
\item
Proposition \ref{torusprop}.\\
The eigenvalues have a form of 
\begin{equation*}
\rho_r(\tau) 
= \exp \Big \{\tau \int_{[0,1]^d}
\frac{
e^{2\pi i \xi\dt r}-1 
}
{|\xi|}\beta(d\xi)
\Big \}.
\end{equation*}
\item
Proposition \ref{qexpansion}.\\
From (\ref{characterize:00})
the eigenvalues have the form
\begin{equation*}
\exp \Big \{\tau\int_{{\cal P}_{\N,\d}}\frac{h_lQ_l(\zeta; (\theta_k)) - 1}{\N-\zeta[0]}\gamma(d\zeta)\Big \},
\end{equation*}
where $\zeta$ are the counts in ${\cal P}_{\N,\d}$, a random partition of $\N$ into $\d$ parts, and $\gamma$ is a measure on these partitions.
The rate matrix, indexed by partitions has elements, for $n\ne m$, of
\begin{equation*}
Q_{nm}=
\p(n;\N)\bigg (
 \sum_{l:|l| \leq \N} \int_{\zeta}\frac{h_lQ_l(\zeta; (\theta_k)) - 1}{\N-\zeta[0]}\gamma(d\zeta)\cdot  h_lQ_l(m;(\theta_k))\overline{Q_l(n;(\theta_k))}\bigg ).
\end{equation*}
The reason why $d-\zeta[0]$ appears in the denominator is to allow for $\gamma$ to have an atom (say $c$) when $\zeta[0]=d$. Then the ratio in (\ref{propcts:00ab}) at zero is interpreted as unity and the integral contributes $c$.
\end{itemize}
\section{The Green function and killing}
A weighted Green function when there is killing at each transition with probability
 $1-\alpha$, $\alpha \in [0,1]$,
for $x,y\in {\cal V}_{\d,\N}$, is
\begin{equation}
G(x,y;\alpha) = \sum_{t=0}^\infty \alpha^t\mathbb{P}\big (X_{t}=y\mid X_0=x\big ).
\label{Green:00}
\end{equation}
\begin{Proposition}\label{grouped:000}
Let $T_\alpha$ be a killing time for $X_t$ such that 
\[
\mathbb{P}\big (T_\alpha = t\big ) = (1-\alpha)\alpha^t,\ t\in \mathbb{N}.
\]
Then 
\begin{equation*}
\mathbb{E}\Big [\mathbb{P}\big (X_{T_\alpha} = y\mid X_0=x \big )\Big] = (1-\alpha)G(x,y;\alpha).
\end{equation*}
\end{Proposition}
\begin{proof}
This is a clear probabilistic connection from (\ref{Green:00}).
\end{proof}
The Green function for $(X_t)$ is now calculated when there is killing.
\begin{Proposition}\label{proposition:expansions}
For $x,y \in {\cal V}_{\d,\N}$,
\begin{align}
(1-\alpha)G(x,y;\alpha) &= 
\frac{1}{\d^\N}\sum_{r\in {\cal V}_{\d,\N}}
\frac{1}{1 + \frac{\alpha}{1-\alpha}(1-\rho_r)}
\theta_1^{(x-y)\dt r}
\label{killed:00}\\
(1-\alpha)\widetilde{G}(m,n;\alpha)&=\frac{1}{\d^\N}
\Bigg \{ 1 + \sum_{l:0 < |l| \leq \N}\frac{h_l}{1 + \frac{\alpha}{1-\alpha}(1-\kappa_l)}Q_l(m;(\theta_k))\overline{Q_l(n;(\theta_k))}\Bigg \}.
\label{killed:00a}\\
\nonumber
\end{align}
In general (\ref{killed:00}) holds, and (\ref{killed:00a}) holds with exchangeability in $m,n$.
In (\ref{killed:00a}) $m \equiv m(x)$, $n \equiv n(y)$ are grouped values in $x,y$ with notation
$\widetilde{G}(m,n;\alpha) \equiv G(x,y;\alpha)$.
\end{Proposition}
\begin{proof}
For (\ref{killed:00})
\begin{align*}
(1-\alpha)G(x,y;\alpha) &= 
\mathbb{P}\big (X_{T_\alpha}=y\mid X_0=x\big )\nonumber\\
&= \sum_{t=0}^\infty (1-\alpha)\alpha^t
\frac{1}{\d^\N}
\sum_{r\in {\cal V}_{\d,\N}}\rho_r^t
\theta_1^{(x-y)\dt r}
\nonumber \\
&=\frac{1}{\d^\N}\sum_{r\in {\cal V}_{\d,\N}}\frac{1}{
1 + \frac{\alpha}{1-\alpha}(1-\rho_r)}\cdot
\theta_1^{(x-y)\dt r}
\end{align*}
The grouped case holds similarly.
\end{proof}
\begin{Remark}
\label{cossin:01}
$(1-\alpha)G(x,y;\alpha)$ is always real so we can take the real part of the right side of (\ref{killed:00}).
Then 
\begin{align}
(1-\alpha)G(x,y;\alpha)
&= \frac{1}{q^d}\sum_{r\in {\cal V}_{\d,\N}}\Re\Big (\frac{1}{ 1 + \frac{\alpha}{1-\alpha}(1-\rho_r)}\Big )
\cos \big( 2\pi/q\ (x-y)\dt r\big )
\nonumber \\
&-\frac{1}{q^d}\sum_{r\in {\cal V}_{\d,\N}}\Im\Big (\frac{1}{ 1 + \frac{\alpha}{1-\alpha}(1-\rho_r)}\Big )
\sin \big( 2\pi/q\ (x-y)\dt r\big ).
\label{real:01}
\end{align}
If the sequence $(\rho_r)$ is real then the second line in (\ref{real:01}) is zero. This observation is important in the construction of real Gaussian fields in Section \ref{GFFSection}.
\end{Remark}
\begin{Remark}\label{PoissonRem}
Take $\tau$ in the embedded process to be a killing time with rate  $\varkappa > 0$.
A calculation shows that with $\alpha=(1+\varkappa)^{-1}$,
\begin{equation*}
\varkappa u^\varkappa(x,y) := \varkappa\int_0^\infty e^{-\varkappa\tau}{\cal P}_\tau(y\mid x)d\tau
= (1-\alpha) G(x,y;\alpha),
\end{equation*}
so 
\begin{equation*}
u^\varkappa(x,y)=
\frac{1}{\varkappa}\cdot\frac{1}{\d^\N}\sum_{r\in {\cal V}_{\d,\N}}\frac{1}{
\big ( 1 + \frac{1}{\varkappa}(1-\rho_r)\big )}
\theta_1^{(x-y)\dt r}.
\end{equation*}
If $Q$ is the rate matrix in the continuous time process then 
\begin{equation*}
(-Q)_{xy} = u^\varkappa(x,y).
\end{equation*}
\end{Remark}
\begin{Proposition}\label{PropB}
Let $(\B_t)$ be a random walk on a $\N$-dimensional torus $\mathbb{T}^\N$ described in Proposition 
\ref{torusprop}.
The killed random walk has a Green function with eigenvectors $e^{2\pi i  \b\dt \r}$ and corresponding eigenvalues of $(1-\alpha)G(\a,\b;\alpha)$ for $\r\in \mathbb{N}^\N$ of
\begin{equation*}
\frac{1}{
1 + \frac{\alpha}{1-\alpha}(1-\rho_\r)},
\end{equation*}
where 
$
\rho_\r = \mathbb{E}\big [e^{2\pi i \V\dt \r}\big ]
$
for $\V$ a wrapped random variable on a $[0,1]^\N$ torus $\mathbb{T}^\N$.
A spectral expansion, which is convergent, is
\begin{equation*}
(1-\alpha)G(\a,\b;\alpha) = 
\sum_{\r \in \mathbb{N}^\N}\frac{1}{1 + \frac{\alpha}{1-\alpha}(1-\rho_\r)}e^{2\pi i (\a-\b)\cdot \r}.
\end{equation*}
\end{Proposition}
\begin{proof}
\begin{align*}
\int_{[0,1]^\N}(1-\alpha)G(\a,\b;\alpha)e^{2\pi i \b\dt\r}d\b
&= (1-\alpha)\sum_{t=0}^\infty \alpha^t 
\int_{[0,1]^\N}e^{2\pi i \b\dt\r}\mathbb{P}\big (\B_t=\b\mid \B_0=\a\big )d\b
\nonumber\\
&= (1-\alpha)\sum_{t=0}^\infty \alpha^t \mathbb{E}\big [e^{2\pi i \V\dt \r}\big ]^t
e^{2\pi i \a \dt\r}
\nonumber\\
&= \frac{1}{
1 + \frac{\alpha}{1-\alpha}(1-\rho_\r)}\cdot e^{2\pi i \a \dt\r}.
\end{align*}
\end{proof}
\begin{Example}\label{Ex:1}
If $V$ is uniform on ${\cal V}_{\d,\N}$ in Proposition 
\ref{proposition:expansions} then $\rho_r= \delta_{rr_0}$, where $r_0$ has all elements zero. The Markov chain then mixes in one step to a uniform distribution on ${\cal V}_{\d,\N}$ and
\begin{equation*}
(1-\alpha)G(x,y;\alpha) = (1-\alpha)\delta_{xy} + \alpha\frac{1}{\d^\N}.
\end{equation*}
An analogous property holds on the $\N$-dimensional torus when $\V$ is uniform on $[0,1]^\N$, then $\rho_\r= \delta_{\r\r_0}$, where $\r_0$ has all entries zero. The Markov chain mixes in one step to a uniform distribution on $[0,1]^\N$ and
\begin{equation*}
(1-\alpha)G(\a,\b;\alpha) = (1-\alpha)\delta_{\a\b} + \alpha.
\end{equation*}
\end{Example}
\subsection{de Finetti exchangeability}\label{deFinetti:sec}
In a de Finetti random walk (\ref{rweqn:000}) on ${\cal V}_{\d,\infty}$ the entries of $Z_t$ are exchangeable taking values in ${\cal V}_q$.
$Z_t$ is controlled by $V_t$ which has a random probability distribution $p$ with de Finetti measure $\mu$.
Take any $\N$ entries within $Z_t$, $Z_t[i_1],\ldots, Z_t[i_d]$ for distinct $i_1,\ldots, i_d$. For given $a_1,\ldots, a_\N \in {\cal V}_\d$ 
\[
\mathbb{P}\big (Z_t[i_1]= a_1,\ldots , Z_{t}[i_\N] = a_\N)
= \int_{ {\cal P}_{ {{\cal V}_\d}}  }p[a_1]\cdots p[a_\N]\mu(dp),
\]
which is invariant under any  $\N$ entries. $ {\cal P}_{ {{\cal V}_\d}} $ is the set of probability measures on ${\cal V}_q$ and $\mu$ is a measure on this space.

Consider a point process of \iid random variables $(V_t)_{t=1}^\infty$ on ${\cal V}_{\d}$. 
We could think of a point process $(\theta_1^{V_t})$ instead of $(V_t)$.
Let $\widehat{X}_0$ have all entries unity. Then 
\begin{equation}
\widehat{X}_t = \odot_{\tau=1}^t\theta_1^{V_\tau},
\label{deft:00}
\end{equation}
where $\odot$ means pointwise multiplication of entries.
The \pgfl of the superposition of $t$ point processes, each with one point,  $\theta_1^{V_1} \ldots , \theta_1^{V_t}$, is
\begin{equation*}
\G[f;(\theta_1^{V_\tau})_{\tau=1}^t ] = \G[f;(\theta_1^{V_\tau})]^t =
\Big (\int_{   {\cal P}_{ {{\cal V}_\d}}}
\sum_{j=0}^{\d-1} f(\theta_1^j)p[j] \mu(dp)\Big )^t
\end{equation*}
That is, the points $\theta_1^{V_1},\ldots, \theta_1^{V_t}$ are  
$\theta_1^{v_1},\ldots, \theta_1^{v_t}$ with probability 
$p_1[v_1]\cdots p_t[v_t]$, where $p_1,\ldots, p_t$ are exchangeable random probability distributions from a de Finetti sequence with measure $\mu$.
The \pgfl of the superposition of points 
$\theta_1^{V_1}, \ldots , \theta_1^{V_{T_\alpha}}$ 
is a mixture
\begin{align}
\mathbb{E}\G[f;(\theta_1^{V_\tau})_{\tau=1}^{T_\alpha}]
&= \sum_{t=0}^\infty(1-\alpha)\alpha^t
\Big (\int_{\Delta_\d}\sum_{j=0}^{\d-1}f(\theta_1^j)p[j]\mu(dp)\Big )^t\nonumber\\[0.1cm]
&=\frac{1}{1 + \frac{\alpha}{1-\alpha}
\int_{{\cal P}_{ {{\cal V}_\d}}}\Big (1 - \sum_{j=0}^{\d-1}f(\theta_1^j)p[j]\Big )\mu(dp)}.
\label{points:01}
\end{align}
The grouped eigenvalues in the spectral expansion (\ref{spectral:001}) are
\begin{equation}
\kappa_l = \int_{{\cal P}_{ {{\cal V}_\d}}}\prod_{k=1}^{\d-1}\Big (\sum_{j=0}^{\d-1}\theta_1^{kj}p[j]\Big )^{l[k]}\mu(dp).
\label{rhoval:00}
\end{equation}
Let $\xi$ be a random variable with entries defined by
\begin{equation}
\xi[k] = \sum_{j=0}^{\d-1}\theta_1^{kj}p[j],\ k \in {\cal V}_\d.
\label{pxi:00}
\end{equation}
Equivalently see (\ref{xidef:2}) below.
Note that $\xi[0]=1$.
If $\mathbb{P}\big (p[j] = p[q-j],\ j\in {\cal V}_q\big )=1$ then $\xi$ is real.
The Fourier transform $p \to \xi$ is $1-1$ because of (\ref{pxi:00}) and the inverse map
\begin{equation*}
p[j] = \frac{1}{\d}\sum_{k=0}^{\d-1}\theta_1^{-kj}\xi[k],\ j \in {\cal V}_\d.
\end{equation*}
Let $\nu$ be the measure on $\xi$ as a change of variable from $p$ and
${\cal X}_\xi$ be the complex state space of $\xi$ which is the set of discrete Fourier transforms of probability distributions in ${\cal P}_{ {{\cal V}_\d}}$.
Then
\begin{equation*}
\kappa_l =
\int_{{\cal X}_\xi} \prod_{k=1}^{\d-1}\xi[k]^{l[k]}\nu(d\xi).
\end{equation*}
\begin{Example}
If $\d=2$ there is only one term $l_1$ in (\ref{rhoval:00}). Denoting the measure of $p[1]$ as $\mu_1$,
\[
\kappa_l = \int_{[0,1]} \big (p[0] - p[1]\big )^{l[1]} \mu_1(dp[1])
 = \int_{[0,1]} \big (1  - 2p[1]\big )^{l[1]}\mu_1(dp[1]).
\]
This is related to the spins in \cite{G2025}.
\end{Example}
There are four types of point processes that are being discussed.
\begin{enumerate}
\item [(a)]
$p_1,\ldots,p_{T_\alpha}$ are points which are random probability
 distributions with a de Finetti measure $\mu$;
\item [(b)]
$V_1,\ldots, V_{T_\alpha}$ are points in ${\cal V}_\d$ chosen from 
$p_1,\ldots,p_{T_\alpha}$;
\item [(c)]\label{points:3}
$\theta_1^{V_1},\ldots, \theta_1^{V_{T_\alpha}}$ is a mapping to the complex roots of $\d$;
\item [(d)]
$\xi_1,\ldots,\xi_{T_\alpha}$ are mean points in ${\cal V}_{\d,\infty}$ from the process in (c), the same as the Fourier transforms
\begin{equation}
\xi_a[k] =  \sum_{j=0}^{\d-1}\theta_1^{kj}p_a[j]=\mathbb{E}_{V_a}\big [\theta_1^{kV_a}\mid p_a\big ],\ a\in [T_\alpha], k \in {\cal V}_\d.
\label{xidef:2}
\end{equation}
\end{enumerate}
\begin{Remark}
It is possible to let $\mu\equiv \mu_\alpha$ and take $\alpha\to 1$, with
\begin{equation*}
\frac{\alpha}{1-\alpha}\mu_\alpha \to \mu^*,
\end{equation*}
where (\ref{points:01}) is still well defined with this limit.
\end{Remark}
\begin{Example}
In the simplest case when $\d=2$, for  $j \in [T_\alpha]$:
\begin{enumerate}
\item [(a)]
$p_j = (p_j[0],p_j[1])$. A random Bernoulli distribution;
\item [(b)]
$V_j$ is Bernoulli($p_j$);
\item[(c)]
\[
\theta_1^{V_j} =
\begin{cases}
+1,&V_j=0\\
-1,&V_j=1 ;
\end{cases}
\]
\item[(d)]
\[
\xi_j[k] =
\begin{cases}
1,&k=0\\
p_j[0]-p_j[1]=1-2p_j[1],&k=1.
\end{cases}
\]
\end{enumerate}
\end{Example}
\medskip

We simplify the expressions that occur from (\ref{xidef:2}) which will lead to expressions of $(1-\alpha)G(x,y;\alpha)$ in terms of moments of random variables.
\begin{Lemma}\label{mainlemma}
Let $(p_a)_{a=1}^\infty$ be an \iid sequence, with each random probability having measure $\mu$. Equivalently let $(\xi_a)_{a=1}^\infty$ be an \iid sequence  with each random variable having probability measure $\nu$. Then
\begin{equation*}
\frac{1}{1 + \frac{\alpha}{1-\alpha}(1-\kappa_l)}
= \mathbb{E}\Big [\prod_{k=1}^{\d-1}Y[k]^{l_k}\Big ],
\end{equation*}
where $Y=(Y[1],\ldots, Y[\d-1])$ are products of points in a $\d-1$ dimensional point process with a joint \pgfl
\begin{align}
\G\big [f_1,\ldots,f_{\d-1}\big ] &=
\mathbb{E}\Big [\prod_{b=1}^{T_\alpha}\int_{{\cal X}_\xi} \prod_{k=1}^{\d-1} f_k\big (\xi_b[k]\big )\nu(d\xi_b)\Big ]\nonumber\\[0.1cm]
&=
\frac{1}{1 + \frac{\alpha}{1-\alpha}\int_{{\cal X}_\xi} \big ( 1 - \prod_{k=1}^{\d-1}f_k(\xi[k])\big ) \nu(d\xi)}.
\label{evlemma:00}
\end{align}
\end{Lemma}
\begin{proof}
Because of independence of the points in $(\xi_a)$
\begin{align*}
\kappa_l^{T_\alpha} &=  \prod_{a=1}^{T_\alpha}\mathbb{E}\big [\prod_{k=1}^{\d-1}\Big (\xi_a[k]\Big )^{l[k]}\mid T_\alpha\Big ]
= \mathbb{E}\big [\prod_{a=1}^{T_\alpha}\Big (\prod_{k=1}^{\d-1}\xi_a[k]\Big )^{l[k]}\mid T_\alpha\Big ]\\
&= \Big (\int_{{\cal X}_\xi} \prod_{k=1}^{\d-1}\xi[k]^{l[k]}\nu(d\xi)\Big )^{T_\alpha},
\end{align*}
so
\begin{equation*}
\mathbb{E}\big [\kappa_l^{T_\alpha}\big ] =
\frac{1}{1 + \frac{\alpha}{1-\alpha}\int_{{\cal X}_\xi} \big ( 1 - \prod_{k=1}^{\d-1}\xi[k]^{l[k]}\big ) \nu(d\xi)}.	
\end{equation*}
The lemma now follows by taking $f_k(\xi[k])= \xi[k]^{l[k]}$.
\end{proof}
In Lemma \ref{mainlemma}, $Y$ is constructed for convenience from $\d-1$ point processes where the points are a partition of points in a single de Finetti sequence. We need a countable partition of a de Finetti sequence 
\begin{equation*}
(\V_j)_{j \in \mathbb{Z}_+}
= \cup_{k=1}^{q-1}(\V_l^{(k)})_{l \in \mathbb{Z}_+},
\end{equation*}
where 
$(\V_l^{(k)})_{l \in \mathbb{Z}_+}\cap (\V_{l^\prime}^{(k^\prime)})_{l^\prime \in \mathbb{Z}_+}
= \emptyset$ if $k\ne k^\prime$. There are many possible such partitions.
Define $Y \in \mathbb{C}^\N$ as having entries $Y[0]=1$ and for $k\in [q-1]$
\begin{equation}
Y[k] = \prod_{j=1}^{T_\alpha}\theta_1^{\V^{(k)}_{j}},
\label{partition:111}
\end{equation}
the product of points in the $k^{\text{th}}$ point process.
\begin{Corollary}
The joint Laplace transform of
$-\log |Y[1]|, \ldots , -\log |Y[\d-1]|$ is
\begin{align}
\mathbb{E}\Big [e^{\sum_{k=1}^{\d-1}-\varphi_k(-\log|Y[k]|)}\Big ]
= \frac{1}{1 + \frac{\alpha}{1-\alpha}\int_{{\cal X}_\xi}\Big (1 - \prod_{k=0}^{q-1}|\xi[k]|^{\varphi_k}\Big )\nu(d\xi)}.
\label{LT:345}
\end{align}
\end{Corollary}
\begin{Example}
Jumps are made one step to the left or right with probability $\gamma/2$ or no step with probability $1-\gamma$ where $\gamma$ is a random variable with a measure we call $\mu$. Then
\begin{equation*}
\kappa_l = \int_{[0,1]}\prod_{k=1}^{\d-1}\big (1-\gamma +\gamma \cos \big (2\pi/\d\cdot k\big )\big )^{l_k}\mu(d\gamma).
\end{equation*}
\begin{itemize}
\item[(a)] $p[0]=1-\gamma$, $p[1]=p[\d-1]=\gamma/2$, where $\gamma$ is random with probability measure $\mu$.
\item[(b)] $V_1,\ldots, V_{T_\alpha}$ take values in ${\cal V}_\d$ with random $\gamma$.
\item[(c)] The transforms $\theta_1^V$ take values $1,\theta_1, \theta_1^{-1}$.
\item[(d)] $\xi[k]=1-\gamma + \gamma\cos(2\pi/\d\cdot k)$,
 $k\in {\cal V}_\d$.
 \end{itemize}
 In Lemma \ref{mainlemma} the \pgfl
 \begin{align*}
\G\big [f_1,\ldots,f_{\d-1}\big ]&=
\frac{1}{1 + \frac{\alpha}{1-\alpha}\int_{[0,1]} \big ( 1 - \prod_{k=1}^{\d-1}f_k(1-\gamma + \gamma\cos(2\pi/\d\cdot k)\big ) \mu(d\gamma)}.
\end{align*}
The Laplace transform (\ref{LT:345}) is 
\begin{align*}
&\mathbb{E}\Big [e^{\sum_{k=1}^{\d-1}-\varphi_k(-\log|Y[k]|)}\Big ]\nonumber\\
&= \frac{1}{1 + \frac{\alpha}{1-\alpha}\int_{[0,1]}1 -
\prod_{k=1}^{\d-1}\big (1-\gamma + \gamma\cos(2\pi/\d\cdot k)\big )^{\varphi_k}\mu(d\gamma)}.
\end{align*}
\end{Example}
\begin{Remark}
Let $T_{\alpha,\phi}$ be a negative binomial random variable with \pgf
$(1-\alpha)^\phi(1-\alpha s)^{-\phi}$, $\phi > 0$. Then
\begin{equation*}
\mathbb{P}\big (T_{\alpha,\phi} = k\big ) = (1-\alpha)^\phi\alpha^k\frac{\phi_{(k)}}{k!},\ k\in \mathbb{N},
\label{Tid:000}
\end{equation*}
where $\phi_{(k)} = \phi(\phi+1)\cdots (\phi+k-1)$. $T_{\alpha,\frac{1}{2}}$ is important in this paper.
A modified proof in Lemma \ref{mainlemma} shows that 
\begin{equation*}
\mathbb{E}\big [\kappa_l^{T_{\alpha,\phi}}\big ] =
\frac{1}
{\Big (1 + \frac{\alpha}{1-\alpha}\int_{{\cal X}_\xi}\big ( 1 - \prod_{k=1}^{d-1}\xi[k]^{l[k]}\big ) \nu(d\xi)\Big )^\phi}.
\end{equation*}   
\end{Remark}
The next lemma is needed in the construction of Gaussian fields.
\begin{Lemma}\label{halflemma}
Let  the entries of $Y$ be the products of points in the process with \pgfl (\ref{evlemma:00}) and the entries of $Y_{\frac{1}{2}}$ and $Y^\prime_{\frac{1}{2}}$ be products of points in two independent processes with \pgfl  
\[
\frac{1}{\Big (1 + \frac{\alpha}{1-\alpha}\int_{{\cal X}_\xi}\big ( 1 - \prod_{k=1}^{\d-1}f_k(\xi[k])\big ) \nu(d\xi)\Big)^{\frac{1}{2}}}.
\]
Then
\begin{equation*}
\mathbb{E}\big [\prod_{k=1}^{q-1} Y_{\frac{1}{2}}[k]^{l[k]}\big ]
\mathbb{E}\big [\prod_{k=1}^{q-1}{Y_{\frac{1}{2}}^{\prime}}[k]^{l[k]}\big ]
= \mathbb{E}\big [\prod_{k=1}^{q-1} Y[k]^{l[k]}\big ].
\end{equation*}
\end{Lemma}
\begin{proof}
Recall that in a multitype \pgfl with $q-1$ point processes 
$(\xi_{j_1}[1]),\ldots, (\xi_{j_{q-1}}[q-1])$,
\[
\G[f_1,\ldots ,f_{q-1}]
= \mathbb{E}\Big [\prod_{k=1}^{q-1}\prod_{j_k=1}^\infty f(\xi_{j_k}[k])\Big ].
\]
Choosing $f_k(\xi[k]) = \xi[k]^{l[k]}$
\[
 \mathbb{E}\big [\prod_{k=1}^{q-1} Y^{l[k]}_{\frac{1}{2}}[k]\big ]
 = \frac{1}
 {\sqrt{1 + \frac{\alpha}{1-\alpha}
 \int_{\chi_\xi}\big ( 1 - \prod_{k=1}^{q-1}\xi[k]^{l[k]}\big )\nu(d\xi)}}
 \]
 and similarly for $Y^\prime$.
 Therefore
 \begin{align*}
 \mathbb{E}\big [\prod_{k=1}^{q-1} Y_{\frac{1}{2}}[k]^{l[k]}\big ]
 \mathbb{E}\big [\prod_{k=1}^{q-1}{Y_{\frac{1}{2}}^{\prime}}[k]^{l[k]}\big ]
&= \frac{1}
 {1 + \frac{\alpha}{1-\alpha}
 \int_{\chi_\xi}\big ( 1 - \prod_{k=1}^{q-1}\xi[k]^{l[k]}\big )\nu(d\xi)}
 \nonumber \\
 &=\mathbb{E}\big [\prod_{k=1}^{q-1} Y[k]^{l[k]}\big ]. 
 \end{align*}
\end{proof}
\begin{Example}
If $p$ has a uniform distribution on ${\cal P}_{ {{\cal V}_\d}}$ with density $\frac{1}{(\d-1)!},\ p \in {\cal P}_{ {{\cal V}_\d}}$, then
\[
\G\big [f_1,\ldots,f_{\d-1}\big ]=
\frac{1}{1 + \frac{\alpha}{1-\alpha}\int_{{\cal P}_{ {{\cal V}_\d}}}\frac{1}{(\d-1)!}
\big ( 1 - \prod_{k=1}^{\d-1}f_k( \sum_{j=0}^{\d-1}\theta_1^{kj}p[j] )  \big )dp}.
\]
\end{Example}
\subsection{Random walk on an infinite-dimensional torus}
Let $(\B_t)$ be a stochastic process on a $\N$-dimensional torus described in Proposition 
\ref{torusprop} with
\begin{equation*}
\rho_\r = \mathbb{E}\big [e^{2\pi i \V\dt \r}\big ],\  \r \in \mathbb{N}^\N.
\label{rho:89}
\end{equation*}
Take the entries of $\V$ to be $\N$ exchangeable random variables from a de Finetti sequence with measure $\mu$ on a space $\varPhi$ such that for distinct $i_1,i_2,\ldots,i_\N \in \mathbb{Z}_+$ the probability measure of 
$\V[i_1],\ldots, \V[i_\N]$ is
\begin{equation*}
\nu_\N(d\V) = \int_\varPhi\otimes_{j=1}^d \nu(d\v_{i_j};\varphi)\ \mu(d\varphi).
\end{equation*}
That is, a mixture of independent measures $\nu^\N(\cdot;\varphi)$ with respect to the de Finetti measure $\mu$ on $\varPhi$. By analogy with (\ref{deft:00}) let 
\begin{equation}
\widehat{\B}_t = \odot_{\tau=1}^te^{2\pi i\V_\tau}.
\label{deft:00a}
\end{equation}
The \pgfl of the superposition of $T_\alpha$ point processes, each with one point, 
$e^{2\pi i\V_1}, \ldots , e^{2\pi i\V_{T_\alpha}}$ is
\begin{align*}
\mathbb{E}\G[f;(e^{2\pi i \V_\tau})_{\tau=1}^{T_\alpha}]
&=\frac{1}{1 + \frac{\alpha}{1-\alpha}
\int_{\varPhi}\int_{[0,1]}\Big (1 - f(e^{2\pi i x})\Big )\nu(dx;\varphi)\mu(d\varphi)}.
\end{align*}
In an analogy of Lemma \ref{mainlemma} and (\ref{partition:111}), we need a countable partition of a de Finetti sequence $(\V_j)_{j \in \mathbb{Z}_+}$ 
into $\cup_{k=1}^\infty(\V^{(k)}_j)_{j\in \mathbb{Z}_+}$.
Define $Y \in \mathbb{C}^\N$ as having entries 
\[
Y[k] = \prod_{j=1}^{T_\alpha}e^{2\pi i\V^{(k)}_{j}},
\]
the product of points in the $k^{\text{th}}$ point process.
\begin{Lemma}
Let $(\B_t)$ be an infinite-dimensional random walk on a torus such that the distribution of any $\N$ entries is described by (\ref{deft:00a}). 
The joint \pgfl of entries in the partitions $\big ( e^{2\pi i  \V^{(k)}_{j}}\big )_{j=1}^{T_\alpha}$, $k\in [d]$ is
\begin{align}
&\G\Big [f_1,f_2,\ldots, f_\N\Big ]\nonumber\\
\nonumber\\
&=\frac{1}{1 + \frac{\alpha}{1-\alpha}
\int_{\varPhi}\Big (1 - \prod_{k=1}^\N\int_{[0,1]}f_k(e^{2\pi i x})\nu(dx;\phi)\Big )\mu(d\varphi)}.
\label{points:01b}
\end{align}
Let $Y$ be such that $Y[k]$ is the product of points in 
$\big ( e^{2\pi i  \V^{(k)}_{j}}\big )_{j=1}^{T_\alpha}$, $k \in [\N]$.
Then in the distribution of $\N$ variables
\begin{equation}
\frac{1}{1 + \frac{\alpha}{1-\alpha}(1- \rho_\r)}
= \mathbb{E}\big [ \prod_{k=1}^\N Y[k]^{\r[k]}\big ].
\label{moments:01}
\end{equation}
\end{Lemma}
\begin{proof}
By definition and the de Finetti measure
\begin{align*}
\rho_\r &= \mathbb{E}\big [\prod_{k=1}^de^{2\pi i \V[k]\r[k]}\big ]
= \int\prod_{k=1}^\N\int_{[0,1]}e^{2\pi i x\r[k]}\nu(dx;\phi)\mu(d\varphi).
\end{align*}
Now the left side of (\ref{moments:01}) is seen to be equal to the \pgfl (\ref{points:01b}) when $f_k(e^{2\pi i x}) = e^{2\pi i x r[k]}$. This means taking the products of points in the $\N$ point processes $Y[k]$ to powers $r[k]$. 
Then (\ref{moments:01}) follows.
\end{proof}
\begin{Proposition}
In an infinite-dimensional de Finetti random walk on a torus the Green function of any $\N$ entries has eigenvector/eigenvalue equations for $\r\in \mathbb{N}^d$ of
\begin{equation*}
(1-\alpha)\int_{\b\in [0,1]^d}e^{2\pi i \b\dt \r}G(\a,\b;\alpha)d\b 
= \mathbb{E}\big [ \prod_{k=1}^\N Y[k]^{\r[k]}\big ]\cdot e^{2\pi i \a\dt r}.
\end{equation*}
\end{Proposition}
\begin{proof}
The proof follows from Proposition \ref{PropB} with the special form (\ref{moments:01}).
\end{proof}
$Y[1], Y[2], \ldots$ form a de Finetti sequence. Let their distribution be described by independent measures  $\nu_Y(\ \cdot \ ;\phi)$ where there is a mixing measure $\mu_Y$ on $\phi$.
Then
\begin{align*}
(1-\alpha)G(\a,\b;\alpha) &= \mathbb{E}\Big [\sum_{\r\in \mathbb{N}^d}
\prod_{k=1}^d Y[k]^{\r[k]}e^{2\pi i (\a[k]-\b[k] )\r[k]}\Big ]
\nonumber\\
&= \mathbb{E}\Big [\prod_{k=1}^d\sum_{\r[k]=0}^\infty \big (Y[k]e^{2\pi i(\a[k]-\b[k])}\big )^{\r[k]}\Big ]\nonumber\\
&= \int_{\Phi}\Big (\prod_{k=1}^d\int_{[0,1]}\Big (\sum_{\r[k]=0}^\infty y^{\r[k]}e^{2\pi i(\a[k]-\b[k])\r[k]}\Big )\nu_Y(dy;\phi)\Big )
\mu_Y(d\phi).
\end{align*}
\section{Gaussian free fields}\label{GFFSection}
In this section we construct several Gaussian fields, indexed by ${\cal V}_{q,d}$, which have covariances related to $G(x,y)$.
The Green function $G(x,y)$ is real and has an expansion (\ref{killed:00}) which contains complex variables.
The eigenvalues $(\rho_r)$ of the underlying random walk have a \emph{Hermitian} property that
\begin{equation}
\bar{\rho}_{q-r} = \mathbb{E}\Big [\prod_{k=1}^d\bar{\theta}_1^{V[k]\dt(q-r)[k]}\Big ]
= \mathbb{E}\Big [\prod_{k=1}^d{\theta}_1^{-V[k]\dt(-r)[k]}\Big ] = \rho_r .
\label{H:00}
\end{equation}
We use a shorthand notation that $q-x$, $x\in {\cal V}_{q,d}$ means  elementwise subtraction and
let $-x\equiv q-x$.
Therefore
\begin{align*}
(1-\alpha)\overline{G(x,y)}&=
\frac{1}{q^d}\sum_{r\in {\cal V}_{q,d}}
\frac{\theta_1^{-(x-y)\dt r}}{1+\frac{\alpha}{1-\alpha}(1-\bar{\rho}_r)}
\nonumber\\
&= \frac{1}{q^d}\sum_{r\in {\cal V}_{q,d}}
\frac{\theta_1^{(x-y)\dt r}}{1+\frac{\alpha}{1-\alpha}(1-{\rho_r})}
= (1-\alpha)G(x,y),
\end{align*}
showing that $G(x,y)$ is real. In the last line the summation index $r\to q-r$.
The Green function $G(x,y)$, $x,y\in {\cal V}_{q,d}$ is exchangeable in $x,y$ if and only if the random walk on ${\cal V}_{q,d}$ is reversible in the model described in Proposition \ref{Proposition:4}.
Reversibility is equivalent to $V=^{\cal D}q - V\ \mod q$, also equivalent to 
$(\rho_r)$ being a real sequence.
The eigenvalues of the random walk are then real.
In this case a complex mean zero Gaussian free field $(g_x)_{x\in {\cal V}_{\d,\N}}$ is constructed by taking
\begin{equation}
\mathbb{E}\big [g_x\overline{g}_y\big ] = (1-\alpha)G(x,y;\alpha).
\label{covdef:00}
\end{equation}
$(g_x)$, being a complex Gaussian field, is not uniquely specified 
by (\ref{covdef:00}). For a unique specification 
Cov$(\Re g_x, \Re g_y)$, Cov$(\Re g_x, \Im g_y)$, Cov$(\Im g_x, \Im g_y)$, are required. 
It is natural to consider the complex Gaussian field in Proposition \ref{fielddef:00}. 
\begin{Proposition}
\label{fielddef:00}
Define a complex Gaussian field on ${\cal V}_{\d,\N}$ by
\begin{equation}
g_x = \frac{1}{\sqrt{\d^\N}}
\sum_{r\in {\cal V}_{\d,\N}}\frac{1}{\sqrt{ 1 + \frac{\alpha}{1-\alpha}(1-\rho_r)     }}
{\theta_1}^{x\dt r}
\ \g_r,
\label{spectral:05}
\end{equation}
where $(\g_r)$ are all \iid N$(0,1)$ or \iid standard complex Gaussian random variables. 
Assume that the entries of $(\rho_r)$ are real.
Then
$
\mathbb{E}\big [g_x\overline{g}_y\big ] = (1-\alpha)G(x,y;\alpha)
$
 with spectral expansion (\ref{killed:00}).
\end{Proposition}
\begin{proof}
This is straightforward calculation using Lemma \ref{halflemma}, and taking care that the entries of $(\rho_r)$ are real.
\end{proof}
If $(\g_r)$ are \iid N$(0,1)$ when $(g_x)$ is defined by (\ref{spectral:05}) then the only non-zero covariances between real and imaginary parts of $g_x,g_y$ are
\begin{align*}
\text{Cov}(\Re g_x, \Re g_y)
&=  \frac{1}{q^d}
\sum_{r\in {\cal V}_{\d,\N}}\frac{1}{\sqrt{ 1 + \frac{\alpha}{1-\alpha}(1-\rho_r)     }}
\cos\big (2\pi/q\ (x-y)\dt r\big )
\nonumber\\
\text{Cov}(\Im g_x, \Im g_y) 
&=  \frac{1}{q^d}
\sum_{r\in {\cal V}_{\d,\N}}\frac{1}{\sqrt{ 1 + \frac{\alpha}{1-\alpha}(1-\rho_r)     }}
\sin \big (2\pi/q\ (x-y)\dt r\big ).
\end{align*}
$\text{\rm Cov}(\Re g_x, \Im g_y) = 0$ so
$(\Re g_x)$ and $(\Im g_y)$ are independent. If $q=2$ then $(q_y)$ is real so $(\Im g_y)$ has zero entries.
\begin{Proposition}
\label{complextoreal:00}
Let $(g_x)$ be defined by (\ref{spectral:05}) where $(\g_r)$ are \iid N$(0,1)$.
Assume that the entries of $(\rho_r)$ are real.
 Then a real Gaussian field $(h^\circ_x)$ with
$
\text{\rm Cov}(h^\circ_x,h^\circ_y) = \mathbb{E}\big [g_x\bar{g}_y \big ] = (1-\alpha)G(x,y;\alpha)
$ is
\begin{align}
h^\circ_x &= \Re g_x + \Im g_x\label{copy:00}\\
&=
\frac{1}{\sqrt{q^d}}\sum_{r\in {\cal V}_{q,d}}
\frac{h^\circ_{x\dt r}}{\sqrt{ 1 + \frac{\alpha}{1-\alpha}(1-\rho_r)}}\g_r,
\nonumber
\end{align}
where $h^\circ_{x\dt r}=\Re \theta_1^{x\dt r} + \Im \theta_1^{x\dt r}=
\cos(2\pi/q\ x\dt r)
+ \sin(2\pi/q\ x\dt r)
$.
$(h_{x\dt r}^\circ)$, $(h_{y\dt r}^\circ)$ are eigenvectors of $G(x,y;\alpha)$ when $(\rho_r)$ is real. Then
\begin{equation}
(1-\alpha)G(x,y;\alpha)
= \frac{1}{q^d}\sum_{r\in {\cal V}_{q,d}}\frac{h^\circ_{x\dt r}h^\circ_{y\dt r}}
{1 + \frac{\alpha}{1-\alpha}(1-\rho_r)}.
\label{Greenreal:00}
\end{equation}

\end{Proposition}
\begin{proof}
We first show that $\Re g_x$ and $\Im g_y$ are independent.
\begin{align}
\text{\rm Cov}\big (\Re g_x, \Im g_y\big )
= \frac{1}{q^d}\sum_{r\in {\cal V}_{q,d}}\frac{1}{1 + \frac{\alpha}{1-\alpha}(1-\rho_r)}
\cos \big (2\pi/q\ x\dt r\big )
\sin \big (2\pi/q\ y\dt r\big ).
\label{A}
\end{align}
Since $\rho_r = \rho_{q-r}$,
\begin{align}
\text{\rm Cov}\big (\Re g_x, \Im g_y\big )
&= \frac{1}{q^d}\sum_{r\in {\cal V}_{q,d}}\frac{1}{1 + \frac{\alpha}{1-\alpha}(1-\rho_{q-r})}
\cos \big (2\pi/q\ x\dt r\big )
\sin \big (2\pi/q\ y\dt r\big )
\nonumber\\
&= \frac{1}{q^d}\sum_{r\in {\cal V}_{q,d}}\frac{1}{1 + \frac{\alpha}{1-\alpha}(1-\rho_{r})}
\cos \big (2\pi/q\ x\dt (q-r)\big )
\sin \big (2\pi/q\ y\dt (q-r)\big )
\nonumber\\
&=\frac{1}{q^d}\sum_{r\in {\cal V}_{q,d}}\frac{1}{1 + \frac{\alpha}{1-\alpha}(1-\rho_{r})}
\cos \big (2\pi/q\ x\dt r\big )
\times -\sin \big (2\pi/q\ y\dt r\big )
\label{B}
\end{align}
In the last two lines of (\ref{B}) the index of summation $r \to q-r$.
Adding (\ref{A}) and (\ref{B}) shows that 
$\text{\rm Cov}\big (\Re g_x, \Im g_y\big )=0$.
The covariance between $h^\circ_x,h^\circ_y$ is correctly related to the Green function.
\begin{align*}
&\text{\rm Cov}(h^\circ_x,h^\circ_y)=\text{\rm Cov}\big (\Re g_x + \Im g_x ,\Re g_y + \Im g_y \big )
= \text{\rm Cov}\big (\Re g_x ,\Re g_y \big )
+ \text{\rm Cov}\big (\Im g_x,\Im g_y \big )
\nonumber\\
&= \frac{1}{q^d}\sum_{r\in {\cal V}_{q,d}}\frac{1}{1 + \frac{\alpha}{1-\alpha}(1-\rho_{r})}
\Big (
\cos(2\pi/q\ x\dt r)\cos(2\pi/q\  y\dt r)
+ \sin(2\pi/q\ x\dt r)\sin(2\pi/q\ y\dt r)
\Big )
\nonumber\\
&= \frac{1}{q^d}\sum_{r\in {\cal V}_{q,d}}\frac{1}{1 + \frac{\alpha}{1-\alpha}(1-\rho_{r})}
\cos(2\pi /q\ (x-y)\dt r)
\nonumber\\
&= (1-\alpha)G(x,y;\alpha).
\end{align*}
The Green function can be written as (\ref{Greenreal:00}) when $(\rho_r)$ is real.
 \begin{align}
 \Re \theta_1^{(y-x)\dt r}
 &= \cos \big (2\pi/q\ x\dt r\big )\cos \big (2\pi/q\ y\dt r\big )
 + \sin \big (2\pi/q\ x\dt r\big )\sin \big (2\pi/q\ y\dt r\big )
 \nonumber\\
 &= \Big (\cos \big (2\pi/q\  x\dt r\big )+ \sin \big (2\pi/q\ x\dt r\big )\Big )
 \Big (\cos \big (2\pi/q\ y\dt r\big )+ \sin \big (2\pi/q\ y\dt r\big )\Big )\nonumber\\
 &-\Big (\cos \big (2\pi/q\ x\dt r\big )\sin \big (2\pi/q\ y\dt r\big )\Big )
 - \Big (\sin \big (2\pi/q\ x\dt r\big )\cos \big (2\pi/q\ y\dt r\big )\Big )
 \nonumber\\
 &= h^\circ_{r\dt x}h^\circ_{ y\dt r} - \Im \theta_1^{(y+x)\dt r}.
 \label{hcalc:01}
 \end{align}
 Therefore
 \begin{equation*}
 (1-\alpha)G(x,y;\alpha) = \frac{1}{q^d}\sum_{r\in {\cal V}_{q,d}}
 \frac{h^\circ_{x\dt r}h^\circ_{y\dt r}}
 {1 + \frac{\alpha}{1-\alpha}(1-\rho_r)}.
  \end{equation*}
 The last line in (\ref{hcalc:01}) does not contribute to the sum in
 $G(x,y;\alpha)$, since $\Im G(-x,y;\alpha) = 0$.
 Also setting $\alpha=0$ shows orthogonality of $(h^\circ_{x\dt r})$ because $G(x,y;0) = \delta_{xy}$.
\end{proof}
If the Green function is not symmetric then $(\rho_r)$ is complex and a Gaussian field cannot be defined by (\ref{covdef:00}). 
Express the eigenvalues of the Green function in Polar coordinate form
\begin{equation*}
1 + \frac{\alpha}{1-\alpha}(1-\rho_r) = R_r\theta_1^{\varphi_r}.
\end{equation*}
The Hermitian property (\ref{H:00}) implies that
\[
R_{q-r}=R_r,\> \varphi_{q-r} = -\varphi_r.
\]
$\theta_1^{\varphi_r}$ is used instead for convenience instead of 
the usual $e^{i\varphi_r}$.

The next proposition shows 
how to extend Proposition \ref{complextoreal:00} to the asymmetric Green function case.
\begin{Proposition}
\label{complextoreal:01}
$(\Re g_x)$ and $(\Im g_y)$, $x,y\in {\cal V}_{q,d}$ are independent.\\
Define
\begin{align*}
g_x^+ &= \frac{1}{\sqrt{q^d}}\sum_{r\in {\cal V}_{q,d}}
\sqrt{R_r^{-1}}\theta_1^{x\dt r + \varphi_r/2}\g_r,\
g_x^-= \frac{1}{\sqrt{q^d}}\sum_{r\in {\cal V}_{q,d}}
\sqrt{R_r^{-1}}\theta_1^{x\dt r - \varphi_r/2}\g_r
\nonumber\\
h_x^+ &= \Re g_x^+ + \Im g_x^+,\ 
h_x^- = \Re g_x^- + \Im g_x^-.
\end{align*}
Then $(\Re g_x^+), (\Im g_y^+)$ and $(\Re g_x^-), (\Im g_y^-)$ are two independent pairs of sequences for $x,y\in {\cal V}_{q,d}$.
$h_x^+$ and $h_x^-$ are real Gaussian fields and Cov$(h_x^+,h_y^-) = (1-\alpha)G(x,y;\alpha)$.
If $(\rho_r)$ is a real sequence then $(\varphi_r = 0)$ and $h_x^+=h_x^-=h_x^\circ$.
\end{Proposition}
\begin{proof} The proof of independence is similar to that using (\ref{A}) and (\ref{B}).
\begin{align}
&\text{\rm Cov}\big (\Re g_x, \Im g_y\big )
\nonumber\\
&= \frac{1}{q^d}\sum_{r\in {\cal V}_{q,d}}R_r^{-1}
\cos\big (2\pi/q[ x\dt r + \varphi_r/2]\big )
\sin\big (2\pi/q [x\dt r + \varphi_r/2]\big )
\label{A:00}\\
&= \frac{1}{q^d}\sum_{r\in {\cal V}_{q,d}}R_{q-r}^{-1}
\cos\big (2\pi/q[ x\dt (q-r)+ \varphi_{q-r}/2]\big )
\sin\big (2\pi/q[ x \dt (q-r)+ \varphi_{q-r}/2]\big )
\nonumber\\
&= \frac{1}{q^d}\sum_{r\in {\cal V}_{q,d}}R_{r}^{-1}
\cos\big (2\pi/q[ x\dt r + \varphi_r/2]\big )
\times
-\sin\big (2\pi/q[ x\dt r + \varphi_r/2]\big )
\label{B:00}
\end{align}
Adding (\ref{A:00}) and (\ref{B:00}) shows that
$\text{\rm Cov}\big (\Re g_x, \Im g_y\big )=0$.
To relate $\text{Cov}(h_x^+,h_y^-)$ to the Green function
\begin{align*}
\text{Cov}(h_x^+,h_y^-)
&=\text{\rm Cov}\big (\Re g_x^+ + \Im g_x^-,\Re g_y^- + \Im g_y^-\Big )
\nonumber\\
&= \text{\rm Cov}\big (\Re g_x^+,\Re g_y^- \big )
+ \text{\rm Cov}\big (\Im g_x^+ , \Im g_y^-\Big )
\nonumber\\
&= \frac{1}{q^d}\sum_{r\in {\cal V}_{q,d}}
R_r^{-1}\Big (\cos\big (2\pi/q[ (x\dt r)+\varphi_r/2]\big )
\cos\big (2\pi/q[ (y\dt r)-\varphi_r/2]\big )
\nonumber\\
&~~~~~~~~~~~~~~+ \sin\big (2\pi/q[ (x\dt r)+\psi_r/2]\big )\sin\big (2\pi/q[ (y\dt r)-\varphi_r/2]\big )\Big )
\nonumber\\
&= \frac{1}{q^d}\sum_{r\in {\cal V}_{q,d}}
R_r^{-1}\cos\big (2\pi/q[(x-y)\dt r + \varphi_r]\big )
\nonumber\\
&= (1-\alpha)G(x,y;\alpha).
\end{align*}
The last line follows because $(1-\alpha)G(x,y;\alpha)$ is real.
\end{proof}
\color{black}
\begin{Remark}
If the entries in $(X_t)$ are exchangeable then $(1-\alpha)G(x,y;\alpha)$ has a spectral expansion (\ref{killed:00a}) with eigenvectors $Q_l(m;(\theta_k))$. Then if $(\rho_j)$ are real
\begin{align}
g_x = 
\frac{1}{q^{d/2}}
\Bigg \{\sum_{l:|l| \leq \N}
\sqrt{\frac{h_l}{1+\frac{\alpha}{1-\alpha}(1-\kappa_l)}}
Q_l(m;(\theta_k))\ \g^\prime_l\Bigg \},
\label{gQ:00}
\end{align}
where 
\begin{equation}
\g_l^\prime = \sqrt{h_l}\sum_{\{k:r[k]=l[j],\ k,j\in {\cal V}_q\}}\g_r.
\label{gprime:00}
\end{equation}
\color{black}
$m\equiv m(x)$ contains the grouped values in $x$. $(\g_l^\prime)$ is an \iid N$(0,1)$ sequence. 
$h^\circ_x = \Re g_x + \Im g_x$ still has the same form as (\ref{copy:00}).
The normalized version of the multivariate Krawtchouk polynomials
are the zonal spherical functions
 \cite{M2004}
\begin{equation}
 h_m m_{\lambda(m)}(\underbrace{1,\ldots,1}_{l_0},\underbrace{\theta_1,\ldots,\theta_1}_{l_1},\cdots,\underbrace{\theta_{q-1},\ldots,\theta_{q-1}}_{l_{q-1}})=h_m Q_l(m;(\theta_k)),
 \label{zonal:00}
\end{equation}
where $\lambda(m)$ is the partition $0^{m_0}1^{m_1}\cdots(q-1)^{m_{q-1}}$
and $m_{\lambda(m)}$ is the monomial symmetric function. An alternative form to (\ref{zonal:00}), (\ref{alternate:00}), arises from direct calculation in the generating function
with $w_0 = 1$
  \begin{align*}
 \Q(m;w,q,d)&=\prod_{j=0}^{q-1}\Big ( \sum_{k=0}^{q-1}\theta_j^kw_k\Big )^{m[j]}
 \nonumber\\
 &= \prod_{j=0}^{q-1}
 \sum_{(m_{jk}): m_{j\cdot}= m[j]} {m[j] \choose (m_{jk})_{k=0}^{q-1}}
 \theta_1^{j\sum_{k=0}^{q-1}km_{jk}}\prod_{k=0}^{q-1}w_k^{m_{jk}}.
 \end{align*}
 The coefficient of $\prod_{k=0}^{q-1}w_k^{l[k]}$ is, 
 with $l[0]=d - \sum_{k=1}^ql[k]$,
 \begin{align}
 Q_l(m;(\theta_k))&= \sum_{(m_{jk}): m_{j\cdot}= m[j], m_{\cdot k}=l[k]} 
   \Big [\prod_{j=0}^{q-1}
  {m[j] \choose (m_{jk})_{k=0}^{q-1}}
  \Big ]
 \theta_1^{\sum_{j,k=0}^{q-1}jkm_{jk}}
 \nonumber\\
 &= \Big (\prod_{j=0}^{q-1}m[j]!\Big )
 \sum_{(m_{jk}): m_{j\cdot}= m[j], m_{\cdot k}=l[k]} 
 \frac{ \theta_1^{\sum_{j,k=0}^{q-1}jkm_{jk}}}
 {\prod_{j,k=1}^{q-1}m_{jk}!}
 \label{alternate:00}
 \end{align}
 To prove (\ref{gQ:00}) first recall the definition of $g_x$ as (\ref{spectral:05}).
  Suppose the grouped configuration of $x$ is $m$. Let $m_{jk} = |\{a:x[a]=j, r[a]=k\}|$.
 \begin{align*}
 g_x &= \frac{1}{q^{d/2}}
 \Bigg \{\sum_{l:|l| \leq d}\frac{\prod_{k=0}^{q-1}l[k]!}{|l|!}
\sqrt{\frac{1}{1+\frac{\alpha}{1-\alpha}(1-\kappa_l)}}
\sum_{\{k:r[k]=l[j],\ k,j\in {\cal V}_q\}}\g_r
\nonumber\\
 &~~~~\times\sum_{(m_{jk}): m_{j\cdot}= m[j], m_{\cdot k}=l[k]} 
   \Big [\prod_{j=0}^{q-1}
  {m[j] \choose (m_{jk})_{k=0}^{q-1}}
  \Big ]
 \theta_1^{\sum_{j,k=0}^{q-1}jkm_{jk}}\Bigg \}
 \nonumber\\
&= \frac{1}{q^{d/2}}
\Bigg \{\sum_{l:|l| \leq d}
\sqrt{\frac{h_l}{1+\frac{\alpha}{1-\alpha}(1-\kappa_l)}}
Q_l(m;(\theta_k))\ \g^\prime_l\Bigg \}.
\end{align*} 
There is a correction for overcounting $l$ in the derivation.
 In this paper we concentrate on the properties of the complex Gaussian field $(g_x)$ when $(\rho_r)$ and $(\g_r)$ are real, knowing that we have a real Gaussian field
 $(h^\circ_x)$, (\ref{copy:00}),  uniquely defined by its covariance function, with
 \[
 \text{\rm Cov}(h^\circ_x,h^\circ_y)
 = \mathbb{E}\big [g_x\bar{g}_y\big ] = (1-\alpha)G(x,y;\alpha).
 \]
In the summation (\ref{gQ:00}) $\kappa_l$ is real and $\kappa_l = \kappa_{l^\prime}$ where $l^\prime[k] = l[q-k]$, $k \in {\cal V}_{q}$. This is the analogue of
$\rho_r=\rho_{q-r}$.
Let $Q^\circ_l(m;(\theta_k))= \Re Q^\circ_l(m;(\theta_k)) 
+ \Im Q^\circ_l(m;(\theta_k))$. A similar Proposition to 
Proposition \ref{complextoreal:01} holds. 
The main parts of the proof are
similar to (\ref{A}) and (\ref{B}). We omitt the proof.
\end{Remark}
\begin{Example}
The simplest real Gaussian field that can be constructed is when
\begin{equation*}
(1-\alpha)G(x,y;\alpha) = (1-\alpha)\delta_{xy} + \alpha\frac{1}{\d^\N}
\end{equation*}
from Example  \ref{Ex:1}. Then
\begin{equation*}
g_x = \sqrt{1-\alpha}\ g_x^* + \sqrt{\alpha\frac{1}{\d^\N}}\ g^*_\emptyset,
\end{equation*}
where $(g^*_x)$ and $g^*_\emptyset$ are independent and N$(0,1)$. This is similar example to that in Corollary 1 of \cite{G2025}.
\end{Example}
\subsection{de Finetti Gaussian field}\label{deFinetti:gaussian}
\begin{Proposition}
In a de Finetti model with $\mathbb{P}\Big (p[j] =^{\cal D}p[q-j],\ j\in {\cal V}_q\Big )=1$, construct a complex mean zero Gaussian free field $(g_x)_{x\in {\cal V}_{\d,\N}}$ by taking
\begin{align}
g_x &= 
\frac{1}{\d^{\N/2}}
\Bigg \{\sum_{l:|l| \leq \N}
\sqrt{\frac{h_l}{1+\frac{\alpha}{1-\alpha}(1-\kappa_l)}}
Q_l(m;(\theta_k))\ \g^\prime_l\Bigg \}\label{strong:a00}\\
&=
\frac{1}{\d^{\N/2}}
\Bigg \{\sum_{l:|l| \leq \N}
\mathbb{E}\Big [\prod_{k=1}^{\d-1}Y_{\frac{1}{2}}[k]^{l[k]}\Big ]
\sqrt{h_l}Q_l(m;(\theta_k))\ \g^\prime_l\Bigg \},
\label{strong:0000}
\end{align}
$m$ contains the grouped values in $x$. $(\g^\prime_l)$ are \iid N$(0,1)$, defined by (\ref{gprime:00}). $(\Re g_x)$ and $(\Im g_y)$ are independent and, with 
$h_x^\circ = \Re g_x + \Im g_x$,
\begin{align}
\text{\rm Cov}(h^\circ_x,h^\circ_y)=\mathbb{E}\big [g_x\overline{g}_y\big ]
&= (1-\alpha)\widetilde{G}(m,n;\alpha)\nonumber\\
&=
\frac{1}{\d^\N}
\Bigg \{\sum_{l:|l| \leq \N}\frac{h_l}{1 + \frac{\alpha}{1-\alpha}(1-\kappa_l)}Q_l(m;(\theta_k))\overline{Q_l(n;(\theta_k))}\Bigg \},
\label{spectral:002}
\end{align}
where $m\equiv m(x)$ and $n\equiv n(y)$ are counts of values in $x,y$.
The first term in the summation in (\ref{strong:0000}) is $\mathfrak{g}_{l_0}$, where $l_0$ contains all zeros. $Y_{\frac{1}{2}}[1],\ldots, Y_{\frac{1}{2}}[\d-1]$ are products of points in a real $\d-1$ dimensional point process with a joint \pgfl
\begin{align}
\G_{\frac{1}{2}}\big [f_1,\ldots,f_{\d-1}\big ] &=
\mathbb{E}\Big [\prod_{b=1}^{T_{\alpha,1/2}}\Big (\prod_{k=1}^\d f_k\big (\xi_b[k]\big )\Big )\Big ]\nonumber\\
&=
\frac{1}{\Big (1 + \frac{\alpha}{1-\alpha}\int_{{\cal X}_\xi}\big ( 1 - \prod_{k=1}^{\d-1}f_k(\xi[k])\big ) \nu(d\xi)\Big)^{\frac{1}{2}}}
\label{evlemma:01}
\end{align}
\end{Proposition}
\begin{proof}
The construction is similar to that in Lemma \ref{mainlemma} taking a killing time to be $T_{\alpha,\frac{1}{2}}$.
All that is needed is to show (\ref{strong:a00}) is that $\mathbb{E}\big [g_x\overline{g}_y\big ]$
is correct, which is straightforward using Lemma \ref{halflemma}. 
Note that $\xi$ is real and $\nu$ a real measure.
\end{proof}
\begin{Remark}
The covariance matrix of $(h^\circ_x)$ has a remarkable form
\begin{align*}
(1-\alpha)G\big (x,y;\alpha\big )
&= \mathbb{E}\Big [
\sum_{r\in {\cal V}_{\d,\N}}\ \prod_{k=0}^{\d-1}Y[k]^{r[k]}\cdot\theta_1^{(x-y)\dt r}\Big ]
\nonumber \\
&= \mathbb{E}\Big [\Big (\sum_{k=0}^{\d-1}Y[k]\theta_1^{x[k]-y[k]}\Big )^\N\Big ].
\label{GN:00}
\end{align*}
\end{Remark}
\subsection{Exchangeable entries in $V$ and spin glasses}
Let $r_\geq= (r[1]_\geq, \ldots, r[\N]_\geq)$ be the ranked values of $r$ from largest to smallest, and 
$\big [r_\geq\big ]$ be the equivalence class of $r \in {\cal V}_{\d,\N}$ with ranked values $r_\geq\ .$
Then
\begin{equation*}
g_x = \frac{1}{\sqrt{\d^\N}}
\sum_{r_\geq}\frac{1}{ \sqrt{ 1 + \frac{\alpha}{1-\alpha}\big (1-\rho_{r_\geq}\big )}}
\cdot\sum_{r \in \big [r_\geq\big ]}{\theta_1}^{\sum_{k=1}^\N r[k]x[k]}
\ \g_r.
\end{equation*}
The last sum in (\ref{spectral:05aA}) is related to spin glass models. When $\d=2$ suppose 
$r_\geq[1]=\cdots r_\geq[j]=1$ and $r_\geq[j+1]=\cdots r_\geq{[d]}=0$. Then
\begin{align}
\sum_{r \in \big [r_\geq\big ]}{\theta_1}^{\sum_{k=1}^\N r[k]x[k]}\ \g_r
=\sum_{i_1,\ldots,i_j\in {\cal V}_\d}(-1)^{x[i_1]+\cdots+x[i_j]}g_{i_1,\ldots,i_j},
\label{sg:00}
\end{align}
where $g_{i_1,\ldots,i_j}$ is a relabelling of $\g_r$.
The Sherrington-Kirkpatrick model is related to the right side of (\ref{sg:00}) when $j=2$.
The general form on the left of (\ref{sg:00}) is an extension from the $2^{\text{th}}$ roots of unity 
$(1,-1)$  to the $q^{\text{th}}$ roots of unity.
\subsection{Gaussian field on a torus}\label{GFTorus}
In the continuous state space model of Proposition \ref{torusprop} define a Gaussian field
 $(g_\b)_{b\in [0,1]^\N}$ with $\mathbb{E}\big [g_\a\bar{g}_\b\big ] = (1-\alpha)G(\a,\b;\alpha)$ for
 $\a,\b \in [0,1]^d$.
We can find a transformation to independent Gaussian variables by knowing the eigenfunctions of 
$(1-\alpha)G(\a,\b;\alpha)$.
\begin{Proposition}
Define
\begin{equation*}
\g_\r = \mbox{$\scriptsize \sqrt{1 + \frac{\alpha}{1-\alpha}\big (1-\rho_r)}$}
\int_{[0,1]^\N}e^{2\pi i\r\dt \b}g_\b d\b,\ \r \in \mathbb{N}^\N,
\end{equation*}
where $\rho_r = \mathbb{E}\big [e^{2\pi i \V\dt\r}\big ]$, $\V$ a random variable on
$[0,1]^\N$.
Then $(\g_\r)_{r\in \mathbb{N}^d}$ are independent standard Gaussian variables.\\
A strong decomposition is
\begin{equation*}
g_\b = \sum_{\r \in \mathbb{N}^\N}
\frac{1}{\sqrt{1 + \frac{\alpha}{1-\alpha}\big (1-\rho_r)}}e^{2\pi i\b\r}\g_\r.
\end{equation*}
\end{Proposition}
\begin{proof}
From Proposition \ref{PropB}
\begin{align*}
\mathbb{E}\big [\g_\r\bar{\g}_\s\big ]
&=  \mbox{$\scriptsize \sqrt{1 + \frac{\alpha}{1-\alpha}\big (1-\rho_r)}$}
 \mbox{$\scriptsize \sqrt{1 + \frac{\alpha}{1-\alpha}\big (1-\rho_s)}$}
 \nonumber \\
&\times \int_{[0,1]^\N}e^{-2\pi i\s\dt\a} \int_{[0,1]^\N}e^{2\pi i\r\dt\b}(1-\alpha)G(\a,\b;\alpha)d\b d\a
\nonumber\\
&= \mbox{$\scriptsize \sqrt{1 + \frac{\alpha}{1-\alpha}\big (1-\rho_r)}$}^{-1}
 \mbox{$\scriptsize \sqrt{1 + \frac{\alpha}{1-\alpha}\big (1-\rho_s)}$}
\int_{[0,1]^\N}e^{-i\s\dt\a}e^{2\pi i\r\dt\a}d\a
\nonumber\\
&= \delta_{\r\s},
\end{align*}
showing independence. The decomposition follows as an eigenfunction expansion.
\end{proof}
\subsection{Central limit form for the Green function as $\N\to \infty$}
Let $M^{(\N)}$ have a uniform Multinomial$(\N,(1/\d))$ distribution. Scale by taking 
$M^{(\N)}= \frac{\N}{\d} +\sqrt{\N}\ \M^{(\N)}$. Then it is well known that 
$\M^{(\N)}$ converges weakly to a singular multivariate normal $\M$ with zero means and
\begin{equation}
 \text{Cov}(\M[a],\M[b]) = \frac{1}{\d}\big (\delta_{ab}- \frac{1}{\d}\big ).
 \label{Mcov:00}
\end{equation}
Removing the first entry
$\mathfrak{M}_+=(\mathfrak{M}[1],\cdots, \mathfrak{M}[\d-1])$ is nonsingular.
The density of $\mathfrak{M}_+$ is, for $\mathfrak{m}_+ \in \mathbb{R}^{\d-1}$
\begin{equation*}
\varphi(\mathfrak{m}_+;\d)=\frac{\d^{\d/2}}{(2\pi)^{(\d-1)/2}}
\cdot \exp\Big \{-\frac{\d}{2}\sum_{a,b=1}^{\d-1}(\delta_{ab}+1)\mathfrak{m}[a]\mathfrak{m}[b]\Big \}.
\end{equation*}
The generating function for the multivariate Krawthouk polynomials (\ref{genfn:00}) can be written in terms of $\mathfrak{m}_+$ as 
\begin{equation*}
\Big (1 + \sum_{k=1}^{\d-1}w_k\Big )^{\N - |\mathfrak{m}_+|}\prod_{j=1}^{\d-1}\Big (1 + \sum_{k=1}^{\d-1}w_k\theta_k^j\Big )^{m[j]}.
\label{genfn:11}
\end{equation*}
Section 3.2 in \cite{GM2025a}) works through limits as $d\to \infty$.
\begin{Lemma}{(Lemma 14, \cite{GM2025a})}
Take a weak limit as $\N\to \infty$ with $M^{(N)}$ uniform multinomial$(\N,(1/\d))$ to obtain a 
multivariate normal $\M$ with covariance function (\ref{Mcov:00}) by setting  
$M^{(\N)}=\N/\d +\sqrt{\N}\ \M$.
Let
\begin{equation}
Q_l(\m,(\theta_j);\infty) =
\lim_{\N\to \infty}Q_l(\N/\d +\sqrt{\N}\ \m;(\theta_j),\N)\N^{-|l|/2},
\label{logcalc:05a}
\end{equation}
where $l\in \mathbb{N}^{\d-1}$. 
$\big (Q_l(\m,(\theta_j);\infty)), (Q_l(\m,(\phi_j);\infty))$ is a biorthogonal system of orthogonal polynomials on $\M$ with 
\begin{equation*}
\mathbb{E}\big [Q_l(\M,(\theta_j);\infty)Q_{l^\prime}(\M,(\phi_j);\infty)\big ]
= \frac{\delta_{ll^\prime}}{l[1]!\cdots l[\d-1]!}.
\end{equation*}
$Q_l(\m,(\theta_j);\infty)$ is the coefficient of $\w^l$ in the generating function
\begin{align*}
\Q(\m,\w,(\theta_j);\infty) =
\exp \Big \{
-\frac{1}{2q}\sum_{j=0}^{q-1}\big(\sum_{k=1}^{\d-1}\w_k\theta_k^j\big)^2
+\sum_{j=0}^{\d-1}\m[j]\sum_{k=1}^{\d-1}\w_k\theta_k^j\Big \}.
\end{align*}
An explicit expression is
\begin{align*}
Q_l(\m,(\theta_j);\infty) &=
\sum_{a\in \mathbb{N}^\d, |a|=|l|}\ \prod_{j=0}^{\d-1}H_{a[j]}(\m[j];\d)\nonumber\\
&\quad\quad\quad\times\frac{1}{l[0]!\cdots l[\d-1]!}
Q_{a^+}(l;(\theta_k), |l|),
\end{align*}
where $a^+ = (a[1],\cdots,a[\d-1])$ and $\big (H_k(x;\d)\big )_{k=0}^\infty$ are Hermite-Chebycheff polynomials orthogonal on a $N(0,1/\d)$ distribution. A generating function for them is
\begin{equation*}
\sum_{k=0}^\infty \frac{1}{k!}z^kH_k(x;\d)=
\exp \Big \{- \frac{1}{2\d}z^2 + xz\Big \}.
\end{equation*}
\end{Lemma}
\begin{Corollary}\label{transformcalc}
The transform
\begin{align*}
&\mathbb{E}\big [e^{i\omega \dt \M}Q_l(m;(\theta_k);\infty)\big ]
=\mathbb{E}\big [e^{i\omega \dt \M}\big ]
\cdot \Big (\frac{i}{q}\Big )^{|l|}\prod_{k=1}^{q-1}\Big (\sum_{a=0}^{q-1}\omega[a]\theta_k^a\Big )^{l[k]}\cdot \frac{1}{l[k]!}.
\end{align*}
for $\omega\in \mathbb{R}^{q}$.
\end{Corollary}
\begin{proof}
Consider the transform of the generating function $\Q(\m,\w,(\theta_j);\infty)$.
\begin{align}
\mathbb{E}\big [e^{i\omega \dt\M}\Q(\M,\w,(\theta_j);\infty)\big ]
&=
\exp \Big \{ -\frac{1}{2\d}\sum_{j=0}^{\d-1}\big (\sum_{k=1}^{\d-1}\w_k\theta_k^j\big )^2
\Big \}
\nonumber\\
&~~~~\times
\mathbb{E}\Big [ e^{i\omega \dt \M}\exp \Big \{\sum_{j=0}^{\d-1}\M[j]
\sum_{k=1}^{\d-1}\w_k\theta_k^j\Big \}\Big ].
\label{calculation:560}
\end{align}
The expectation factor in the right of (\ref{calculation:560}) is 
\begin{align}
&\mathbb{E}\big [ e^{i\omega \dt\M}\exp \Big \{\sum_{j=0}^{\d-1}\M[j]\sum_{k=1}^{\d-1}\w_k\theta_k^j\Big \}\big ]
\nonumber\\
&= \exp \Big \{ 
-\frac{1}{2}\sum_{a=0}^{q-1}\sum_{b=0}^{q-1}
\big (\omega[a]-i\sum_{k=1}^{\d-1}\w_k\theta_k^a\big )
\big (\omega[b]-i\sum_{k=1}^{\d-1}\w_k\theta_k^b\big )
\frac{1}{q}\big (\delta_{ab} - \frac{1}{q}\big ) 
\Big \}\nonumber\\
&= \exp \Big \{-\frac{1}{2q}\sum_{a=0}^{q-1}\Big (\big (\omega[a]-i\sum_{k=1}^{\d-1}\w_k\theta_k^a\big )\Big )^2 \Big \}
\label{calculation:720a}\\
&~~~\times
\exp \Big \{
\frac{1}{2q^2}\Big (\sum_{a=0}^{q-1}\big (\omega[a]-i\sum_{k=1}^{\d-1}\w_k\theta_k^a\big )\Big )^2\Big \}.
\label{calculation:720}
\end{align}
The exponent in (\ref{calculation:720a}) is equal to
\begin{equation*}
-\frac{1}{2q}\sum_{a=0}^{q-1}\omega[a]^2
+ \frac{i}{q}\sum_{k=1}^{q-1}\sum_{a=0}^{q-1}\omega[a]\theta_k^a\w_k
+ \frac{1}{2q}\sum_{a=0}^{q-1}\Big (\sum_{k=1}^{q-1}\w_k\theta_k^a\Big )^2 
\end{equation*}
and the exponent in (\ref{calculation:720}) simplifies to
$
\frac{1}{2q^2}\Big (\sum_{a=0}^{q-1}\omega[a]\Big )^2.
$
Thus
\begin{align*}
&\mathbb{E}\big [e^{i\omega \dt \M}\Q(\M,\w,(\theta_j);\infty)\big ]
\nonumber\\
&=
\exp \Big \{-\frac{1}{2q}\sum_{a=0}^{q-1}\omega[a]^2 +\frac{1}{2q^2}\big (\sum_{a=0}^{q-1}\omega[a]\big )^2
+ \frac{i}{q}\sum_{k=1}^{q-1}\sum_{a=0}^{q-1}\omega[a]\theta_k^a\w_k\Big \}.
\end{align*}
Then Corollary \ref{HCgen:00} follows by noting that
\[
\mathbb{E}\big [e^{i\omega \dt \M}\big ]
= \exp \Big \{-\frac{1}{2q}\sum_{a=0}^{q-1}\omega[a]^2 +\frac{1}{2q^2}\big (\sum_{a=0}^{q-1}\omega[a]\big )^2\Big \},
\]
and equating coefficients of $\prod_{k=1}^{q-1}\w_k^{l[k]}$.

A complex argument is used for the characteristic function of a multivariate normal distribution in (\ref{calculation:720a}), (\ref{calculation:720}). An elementary calculation shows that this is allowable.
\end{proof}
\begin{Corollary}\label{Greenlimit}
$d^{-(q-1)}{\N\choose \m}{\N\choose \n}\frac{1}{\d^\N}(1-\alpha)\widetilde{G}(m,n;\alpha)$, expressed in Proposition \ref{proposition:expansions} as proportional to (\ref{killed:00a}) converges to
\begin{align}
&(1-\alpha)\widetilde{G}_\infty(\m,\n,\alpha)=\varphi(\mathfrak{m}_+;\d)\varphi(\mathfrak{n}_+;\d)\nonumber\\
&\times\Big \{ 1 + \sum_{l \in \mathbb{N}^{\d-1}, |l| \ne 0}\frac{\prod_{j=1}^{\d-1}l[j]!}{1 + \frac{\alpha}{1-\alpha}(1-\kappa_l)}
Q_l(\mathfrak{m},(\theta_k);\infty)\overline{Q_l(\mathfrak{n},(\theta_k),\infty)}
\Big \},
\label{limitG:00}
\end{align}
when $x=\frac{\N}{\d} + \sqrt{\N}\ \mathfrak{m}$, $y=\frac{\N}{\d} + \sqrt{\N}\ \mathfrak{n}$ as
$\N\to \infty$.
\end{Corollary}
\begin{proof}
\[
d^{-(q-1)/2}{d\choose \m}\frac{1}{q^d} \to \varphi(\m_+;q)
\]
and 
$
h_l \sim d^{-|l|}\prod_{j=1}^{q-1}l[j]!
$
so (\ref{limitG:00}) follows from (\ref{logcalc:05a}).
\end{proof}
\begin{Proposition}\label{gfclt:00}
Let $V\in {\cal V}_{q,d}$ be from a de Finetti sequence. Assume that $V$ takes symmetric jumps.
Define a Gaussian field depending on $m_+$ as
\begin{equation*}
g_{m_+}=
\p(m;\N)
\Bigg \{ \sum_{l:|l| \leq \N}
\sqrt{\frac{h_l}{1 + \frac{\alpha}{1-\alpha}(1-\kappa_l)}}
Q_l(m;(\theta_k))\ \g_l\Bigg \}.
\label{gfdef:87}
\end{equation*}
A limit Gaussian field on $\mathbb{R}^{\d-1}$ is 
\begin{align*}
g_{\mathfrak{m}_+}^\infty &=
\lim_{\N\to \infty;\ m= \N/\d +\sqrt{\N}\ \m}d^{-(q-1)/2}g_{m_+} 
\nonumber\\
&=\varphi(\mathfrak{m}_+;\d)
\Bigg \{ \sum_{|l| \in \mathbb{N}^{\d-1}}
\mathbb{E}\Big [\prod_{k=1}^{\d-1}Y_{\frac{1}{2}}[k]^{l[k]}\Big ]
\sqrt{\prod_{j=1}^{\d-1}l[j]!}\
Q_l(\mathfrak{m};(\theta_k);\infty)\ \g_l\Bigg \},
\end{align*}
where the entries of $Y_{\frac{1}{2}}$ are products of points in a $\d-1$ dimensional point process with a joint \pgfl (\ref{evlemma:01}).
\begin{align}
&\mathbb{E}\Bigl [g_{\mathfrak{m}_+}^\infty
\bar{g}_{\mathfrak{n}_+}^\infty\Bigr ]= (1-\alpha)\widetilde{G}_\infty(\m,\n,\alpha)
\nonumber\\
&=\varphi(\mathfrak{m}_+;\d)\varphi(\mathfrak{n}_+;\d) \nonumber\\
&~~\times\Big \{\sum_{|l| \in \mathbb{N}^{\d-1}}
\mathbb{E}\Big [\prod_{k=1}^{\d-1}Y[k]^{l[k]}\Big ]
\prod_{j=1}^{\d-1}l[j]!
Q_l(\mathfrak{m};(\theta_k);\infty)\overline{Q_l(\mathfrak{n};(\theta_k);\infty)}
\Bigg \},
\label{cov:55}
\end{align}
where the entries of $Y$ are products of points in a $\d-1$ dimensional point process with a joint \pgfl (\ref{evlemma:00}). $(\Re g_{m_+})$ and $(\Im g_{n_+})$ are independent and
with $h_{\m_+} = \Re g_{m_+} + \Im g_{m_+}$,
\[
\text{\rm Cov}(h_{\mathfrak{m}_+},h_{\mathfrak{n}_+}) = (1-\alpha)\widetilde{G}_\infty(\m,\n,\alpha).
\]
\end{Proposition}
\begin{proof}
The main point is to understand in the proof is the expression of the eigenvalues in terms of $Y, Y_{\frac{1}{2}}$. Recall that (\ref{spectral:002}), (\ref{strong:a00}) and (\ref{strong:0000}) hold. There is an unusual property that
\begin{equation*}
\Big (\mathbb{E}\Big [\prod_{k=1}^{q-1}Y_{\frac{1}{2}}[k]^{l[k]}\Big ]\Big )^2
= \mathbb{E}\Big [\prod_{k=1}^{q-1}Y[k]^{l[k]}\Big ].
\end{equation*}
This occurs because of the form of the \pgfl (\ref{evlemma:01}) (see Lemma \ref{halflemma}). Squaring 
(\ref{evlemma:01}) shows that $Y$ has the same distribution  as
$Y_{\frac{1}{2}}\odot Y^\prime_{\frac{1}{2}}$, where $Y_\frac{1}{2}$ and $Y^\prime_\frac{1}{2}$ are independent.
\end{proof}
\subsection{Gaussian field of a transform}
\begin{Proposition} \label{gftransform:00}
Define a zero mean complex Gaussian field for $\omega\in \mathbb{R}^{d-1}$ by
\begin{align*}
&\widehat{g}_{\omega}=
\int_{\mathbb{R}^{d-1}}
e^{i\omega\dt \mathfrak{m}_+}g_{\mathfrak{m}_+}^\infty\, 
\varphi(\mathfrak{m}_+;d)d \mathfrak{m}_+\nonumber\\
&= \mathbb{E}\big [e^{i\M_+\dt\omega}\big ]
\mathbb{E}_{ Y_{\frac{1}{2}} }\Bigg \{
\sum_{|l| \in \mathbb{N}^{\d-1}}
\Big (\frac{i}{\d}\Big )^{|l|}\prod_{k=1}^{\d-1}
\frac{
\Big (Y_{\frac{1}{2}}[k]\sum_{j=1}^{\d-1}\omega[j]\theta_k^j\Big )^{l[k]}
}
{\sqrt{l[k]!}}\ \cdot\mathfrak{g}_l\Bigg \},
\end{align*}
where the entries of $Y_{\frac{1}{2}}$ are products of points in a $\d-1$ dimensional point process with a joint \pgfl (\ref{evlemma:01}).
Assume that $Y$ is real.
Then
\begin{align}
&\mathbb{E}\Big [\widehat{g}_{\omega}\overline{\widehat{g}}_{\psi}\Big ]=
\int_{\mathbb{R}^{d-1}}\int_{\mathbb{R}^{d-1}}
e^{i\omega\dt \mathfrak{m}_+}e^{-i\psi\dt \mathfrak{n}_+}
\widetilde{G}_\infty(\m,\n;\alpha)d\m d\n
\nonumber\\
&= \mathbb{E}\big [e^{i\M_+\dt\omega}\big ] \mathbb{E}\big [e^{i\M_+\dt\psi}\big ] 
\mathbb{E}\Big [
\exp \Big \{\frac{1}{\d^2}\sum_{k=1}^{\d-1}Y[k]\sum_{a,b=1}^{\d-1}
\omega[a]\psi[b]\theta_k^a\theta_k^{-b}
\Big \}\Big ],
\label{complexcov:000}
\end{align}
where
\begin{equation*}
\mathbb{E}\big [e^{i\M_+\dt\omega}\big ] 
= \exp \Big \{-\frac{1}{2q}\sum_{a,b=1}^{\d-1}\big (\delta_{ab}- \frac{1}{\d}\big )\omega[a]\omega[b]\Big \}.
\label{mgf:256}
\end{equation*}
\end{Proposition}
\begin{proof}
Taking the double transform of (\ref{cov:55}) directly and using Corollary \ref{transformcalc} with $\omega[0]=\psi[0]=0$,
\begin{align*}
\mathbb{E}
\Big [\widehat{g}_{\omega}\overline{\widehat{g}}_{\psi}\Big ]
&= \int_{\mathbb{R}^{d-1}}\int_{\mathbb{R}^{d-1}}
e^{i\omega\dt \mathfrak{m}_+ - i \psi\dt \mathfrak{n}_+}
\mathbb{E}\Big [g_{\mathfrak{m}_+}^\infty \bar{g}_{\mathfrak{n}_+}^\infty \Big ]dm_+dn_+
\nonumber\\
&=\sum_{|l| \in \mathbb{N}^{\d-1}}
\mathbb{E}\Big [\prod_{k=1}^{\d-1}Y[k]^{l[k]}\Big ]
\prod_{j=1}^{\d-1}l[j]!\nonumber\\
&~~~~~\times\mathbb{E}\big [e^{i\omega\dt \mathfrak{m}_+}
Q_l(\mathfrak{m};(\theta_k);\infty)\big ]
\mathbb{E}\big [e^{- i \psi\dt \mathfrak{n}_+}
\overline{Q_l(\mathfrak{n};(\theta_k);\infty)}\big ]
\nonumber\\
&=\sum_{|l| \in \mathbb{N}^{\d-1}}
\mathbb{E}\Big [\prod_{k=1}^{\d-1}Y[k]^{l[k]}\Big ]
\prod_{j=1}^{\d-1}l[j]!\nonumber\\
&~~~~~\times
\mathbb{E}\big [e^{i\omega \dt \M_+}\big ]
\cdot \Big (\frac{i}{q}\Big )^{|l|}\prod_{k=1}^{q-1}\Big (\sum_{a=1}^{q-1}\omega[a]\theta_k^a\Big )^{l[k]}\cdot \frac{1}{l[k]!}
\nonumber\\
&~~~~~\times
\mathbb{E}\big [e^{-i\psi \dt \M_+}\big ]
\cdot \Big (\frac{-i}{q}\Big )^{|l|}\prod_{k=1}^{q-1}\Big (\sum_{b=1}^{q-1}\psi[b]\theta_k^{-b}\Big )^{l[k]}\cdot \frac{1}{l[k]!}
\nonumber\\
&=\mathbb{E}\big [e^{i\M_+\dt\omega}\big ] \mathbb{E}\big [e^{i\M_+\dt\psi}\big ] 
\mathbb{E}\Big [
\exp \Big \{\frac{1}{\d^2}\sum_{k=1}^{\d-1}Y[k]\sum_{a,b=1}^{\d-1}
\omega[a]\psi[b]\theta_k^a\theta_k^{-b}
\Big \}\Big ].
\end{align*}
 \end{proof}
The covariance function (\ref{complexcov:000}) agrees with the covariance function for $q=2$ in Proposition 10 of \cite{G2025}. 
 \subsection{Gaussian Field and Hamiltonians}\label{Hamiltonian:00}
 We calculate the Hamiltonian based on $G^{-1}(x,y;\alpha)$ and
 $(h^\circ_x)$, $(h^\circ_y)$.
 $G(x,y;\alpha)$ and $G^{-1}(x,y;\alpha)$ are real which leads to an expansion in real variables instead of $(\theta_1^{r\dt(y-x)})$ when $(\rho_r)$ is real.
  $G(x,y;\alpha)$ and  $G^{-1}(x,y;\alpha)$ have the same eigenvectors $(h^\circ_{r\dt x})$.
 Define
 \begin{align*}
 h_x^\circ(\alpha) = \sqrt{ \frac{\alpha}{1-\alpha}}h_x^\circ 
=\frac{1}{\sqrt{q^d}}
 \sum_{r \in {\cal V}_d}
\frac{1}{\sqrt{1/\alpha - \rho_r}}
h^\circ_{r\dt x}\, \g_r.
 \end{align*}
The Hamiltonian for $\big (h^\circ_x(\alpha)\big )$ is 
\begin{align}
{\cal H} =
\frac{1}{2\alpha}\sum_{x,y \in {\cal V}_{q,d}}G^{(-1)}(x,y;\alpha)h^\circ_x(\alpha)h^\circ_y(\alpha)
= \frac{1}{2}\sum_{r \in {\cal V}_{q,d}}\g_r^2.
\label{Hamiltonian:66a}
\end{align}
In the calculation of (\ref{Hamiltonian:66a}) note that $\big (G(x,y;\alpha)\big )$ and $\big (G^{(-1)}(x,y;\alpha)\big )$ have the same eigenvectors $(h_{r\dt x}^\circ),(h_{r\dt y}^\circ)$, but eigenvalues $(1-\alpha\rho_r)^{-1}$ and $(1-\alpha\rho_r)$.

We need to find the Jacobian in a change of variable from
 $(h^\circ_x(\alpha))$ to $(\g_r)$. Now
 $\big(\frac{1}{q^{d/2}}\cdot h^\circ_{r\dt x}\big )$ is an orthogonal matrix so has a determinant of unity.
 \begin{align*}
J=\Big |\frac{\partial \big (h^\circ_x(\alpha)\big )}{\partial (\g_r)}\Big |
&= \Big |\Big ( \frac{1}{\sqrt{1/\alpha - \rho_r}}\cdot\frac{1}{q^{d/2}}\cdot h^\circ_{x\dt r}
\Big )\Big |
\nonumber\\
&= 
\prod_{r \in {\cal V}_{q,d}} \frac{1}{\sqrt{1/\alpha - \rho_r}}
\cdot \Big |\big (\frac{1}{q^{d/2}}\cdot h^\circ_{x\dt r}\big )\Big |
\nonumber \\
&= 
\alpha^{q^d/2}\prod_{r \in {\cal V}_{q,d}} \frac{1}{\sqrt{1 - \alpha\rho_r}}.
 \end{align*}
 The partition function is, for $\beta > 0$,
 \begin{equation*}
Z=\int_{\mathbb{R}^{q^d}}e^{-\beta{\cal H}}d(h^\circ_x(\alpha))
= (2\pi/\beta)^{q^d/2}J
 \end{equation*}
The behaviour of $J$ depends on the model for $\V$. We assume that the entries of $\V$ are exchangeable. Then $(\rho_r)$ can be grouped into $(\kappa_l)$ with 
${d\choose |l|}{|l|\choose l}$ elements from the former sequence contributing to $\kappa_l$. Suppose the grouped values of $\V$ are $n_{\V}$ so there are $n_{\V}[k]$ entries in $\V$ equal to $k$. We know that
\begin{equation}
\kappa_l=h_l\mathbb{E}\big [Q_l(n_\V;(\theta_j))\big ].
\label{kappa:00}
\end{equation}
There is a duality relationship for the multivariate Krawtchouk polynomials, in \cite{GM2025a}, Proposition 9.
Denote $l^+=(\N-|l|,l_1,\ldots,l_{\d-1})$ and 
$m^- = (m[1],\ldots, m[\N-1])$. Then
\begin{equation*}
h_{m^-}^{-1}Q_l(m; (\theta_k)) = h_{l}^{-1}Q_{m^-}(l^+;(\theta_j)).
\label{MVKdual:00}
\end{equation*} 
Therefore another expression for (\ref{kappa:00}) is
\begin{equation*}
\kappa_l = \mathbb{E}\big [h_{\V^-} Q_{\V^-} (l^+;(\theta_j))\big ].
\label{kappa:01}
\end{equation*}
In a multinomial grouping there are ${d\choose |l|}{|l|\choose l}$ terms $\rho_r$ contributing to $\kappa_l$. It is helpful to think of
\begin{align}
&-\frac{1}{2}\sum_{l :|l| \leq d}{d\choose |l|}{|l|\choose l}q^{-d}\log(1-\alpha\kappa^{(d)}_l) 
\nonumber\\
&= -\frac{1}{2}\sum_{l^+\in {\cal P}_{d,q}}{\N\choose l^+}q^{-d}\log(1-\alpha\kappa^{(d)}_l)
\label{mult:133}
\end{align}
as a multinomial expectation over $l^+\in {\cal P}_{d,q}$ with a multinomial distribution ${\N\choose l^+}q^{-d}$.
This leads to a way to find limit distributions from a weak limit as $d\to \infty$.
In the sum (\ref{mult:133}) by the weak law of large numbers $l^+$ converges to having all entries of $d/q$. In this limit we are assuming that 
$\kappa_l \equiv \kappa_l^{(d)}$ and it is possible to choose $\V^{(d)}$ so that 
$\kappa^{(d)}_{(d/q)}$ converges properly. Then the limit (\ref{mult:133}) will be equal to
\begin{equation*}
-\frac{1}{2}\lim_{d\to \infty}\log (1-\alpha\kappa^{(d)}_{(d/q)})
\end{equation*}
Assume that $d^{-1}\V^{(d)}[k] \to \Z[k]$, $k=1,\ldots,q-1$, where $\Z$ is random, in the sense of weak convergence. From Lemma 2 in \cite{GM2025a} 
conditional on $\Z=\z$
\begin{align*}
\lim_{d\to \infty}\kappa^{(d)}_{(d/q)}&=
\exp \Big \{\frac{1}{q}\sum_{k=1}^{q-1}(-|\z| + \sum_{j=1}^{q-1}\z[j]\theta_k^j)\Big \}
\nonumber\\
&=\exp \Big \{-\frac{2q-1}{q}|\z|\Big \}.
\end{align*}
Therefore
\begin{equation*}
-\frac{1}{2}\lim_{d\to \infty}\log (1-\alpha\kappa^{(d)}_{(d/q)})
= -\frac{1}{2}\mathbb{E}\Big [\log \Big ( 1 - \alpha e^{-\frac{2q-1}{q}|\Z|}\Big )\Big ].
\end{equation*}
Finally 
\begin{align*}
\log Z &\sim \frac{q^d}{2}\log(2\pi) + \frac{q^d}{2}\log\Big ( \frac{1}{\beta}\Big )
+ \frac{q^d}{2}\log \alpha
- \frac{q^d}{2}\mathbb{E}\Big [\log \Big ( 1 - \alpha e^{-\frac{2q-1}{q}|\Z|}\Big )\Big ].
\end{align*}
\begin{equation}
\lim_{d\to \infty}\frac{2}{q^d}\log Z
= \log(\frac{2\pi\alpha}{\beta})
+\mathbb{E}\Big [-\log \Big ( 1 - \alpha e^{-\frac{2q-1}{q}|\Z|}\Big )\Big ].
\label{final:88}
\end{equation}
Considering the behaviour of the expectation on the right of (\ref{final:88})
$\lim_{\alpha \to 1}\lim_{d\to \infty}\frac{2}{q^d}\log Z$ can be finite or infinite.
Suppose that the distribution of $|\Z|$ does not depend on $\alpha$.
As $\alpha\to 1$, $\z\to 0$, 
\[
-\log \Big ( 1 - \alpha e^{-\frac{2q-1}{q}|\Z|}\Big ) \sim
- \log\big (|\Z|\big ).
\]
Therefore 
\begin{equation*}
\lim_{\alpha \to 1}\lim_{d\to \infty}\frac{2}{q^d}\log Z < \infty
\end{equation*}
if and only if $\mathbb{E}\big [-\log |\Z|\big ] < \infty$.\\[0.2cm]
\subsection{Potts model}
\label{Pottssection}
There is a huge literature on the Potts model \cite{P1952,W1982}.
A random bond model has a Hamiltonian of the form
\begin{equation}
{\cal H}(\theta) = -\sum_{s\leq d} a_s\sum_{i_1,\ldots,i_s}
J_{i_1,\ldots,i_s}\delta_{K_s}(\theta_{i_1},\ldots, \theta_{i_s})
\label{Potts:00}
\end{equation}
In \eqref{Potts:00} the indices $i_1,\ldots,i_s$  denote the positions of the 1st to $s$-th spins with $\theta_{i_j}\in\{\theta_1^i:i=0,1,\ldots,q-1\}$;
$(a_s)$ are constants;  $(J_{i_1,\ldots, i_s})$ are random variables;
 $\delta_{K_s}(\theta_{i_1},\ldots, \theta_{i_s})=1$ if
 $\theta_{i_1}= \cdots = \theta_{i_s}$ or zero otherwise; and the sum is over the particular geometry being considered. Interest is in how the behaviour varies with the probability distribution of $(J_{i_1,\ldots,i_s})_{s\leq d}$. 
The energy is minimized when all spins take the same value. 
 Usually $d=2$, but we consider general $d$.\\[0.2cm]
In our context we can consider a Hamiltonian, for $y\in {\cal V}_{q,d}$
\begin{equation*}
  {\cal H}_y(\theta) = -\sum_{x\in {\cal V}_{q,d}}b(y,x)g_x
  = - \sum_{r\in {\cal V}_{q,d}}J_r
  \sum_{x\in {\cal V}_{q,d}}b(y,x)\theta_1^{x\dt r}, 
\end{equation*}
where
\[
J_r=
  \frac{1}{\sqrt{q^d}}\frac{\g_r}{\sqrt{1+\frac{\alpha}{1-\alpha}(1-\rho_r)}};
\]
$b(y,x) \in \mathbb{C}$, and could be random variables;
$y$ specifies the spin configuration of the system; and $J_r$
are random bonds among spins determined by a random walk.
\begin{Remark}
  The physics interpretation of our Hamiltonian is as follows.
  The spins live in $d$ positions,
  $(\theta_1^{y_1},\ldots,\theta_1^{y_d})$, where
  $y=(y_1,\ldots,y_d)$ specifies the state. 
  If $q=3$, $d=2$, and $b(y,x)=\delta_{xy}$,
  \begin{align*}
  \mathcal{H}_y(\theta)
  =&-J_{00}-J_{01}(\theta_1^{y_2})+J_{02}(\theta_1^{y_2})^2
  -J_{10}(\theta_1^{y_1})-J_{11}(\theta_1^{y_1})(\theta_1^{y_2})
  -J_{12}(\theta_1^{y_1})(\theta_1^{y_2})^2\\
  &-J_{20}(\theta_1^{y_1})^2
  -J_{21}(\theta_1^{y_1})^2(\theta_1^{y_2})
  -J_{22}(\theta_1^{y_2})^2(\theta_1^{y_2})^2
  \end{align*}
  for $(\theta_1^{y_1},\theta_1^{y_2})$, where $y_1,y_2 \in \{0,1\}$ and $\theta_1=-1$. 
  Here $J_{r_1,r_2}$ determines the orders of the interaction
  between spins at the first and the second position.
  In other words, the Hamiltonian is regressed on the spins
  $(\theta_1^{y_1},\theta_1^{y_2})$ allowing
  higher-order interactions.  
\end{Remark}
It is straightforward to see that
\[
  \mathbb{E}\big[\mathcal{H}_y(\theta)\big ]=0, \quad
  \mathbb{E}\big [|\mathcal{H}_y(\theta)|^2\big ]=\sum_{x,x'\in {\cal V}_{q,d}}
  b(y,x)\overline{b(y,x')}{\rm Cov}(x,x').
\]
If $q=2$ and $b(y,x)=\delta_{xy}$,
$\mathcal{H}_y(\theta)=-\sum_{r\in{\cal V}_{2,d}} J_r (-1)^{y\dt r}$,
and we observe $\mathcal{H}_y(\theta)\sim {\rm N}(0,1)$ if $\alpha=0$, while
$\mathcal{H}_y(\theta)\sim {\rm N}(0,1/2^d)$ as $\alpha\to 1$.
Fluctuation of the Hamiltonian increases with killing.

Interest is in the random probability distribution in $y$
\begin{equation*}
\mathbb{P}(y;\beta)=
e^{\beta {\cal H}_y(\theta)}/Z,\ y \in {\cal V}_{q,d}
\end{equation*}
where $\beta > 0$ is a constant and 
\[
Z = \sum_{y\in {\cal V}_{q,d}}e^{\beta {\cal H}_y(\theta)}
\]
is the partition function.
If $b(y,x)$ is non-random,
\begin{align*}
\mathbb{E}\big [Z\big ] &= \sum_{y\in {\cal V}_{q,d}} \prod_{r\in {\cal V}_{q,d}} \mathbb{E}
\big[ e^{\beta J_r \sum_{x\in{\cal V}_{q,d}} b(y,x)\theta_1^{x\dt r}}\big]\\
  &= \sum_{y\in {\cal V}_{q,d}} \prod_{r\in {\cal V}_{q,d}}
\exp\left\{\frac{\beta^2}{2q^d}
\frac{\big (\sum_{x\in{\cal V}_{q,d}}b(y,x)\theta_1^{x\dt r}\big )^2}{1+\frac{\alpha}{1-\alpha}(1-\rho_r)}
\right\}.
\end{align*}
Moreover, if $q=2$ and $b(y,x)=\delta_{xy}$,
we have
\[
\log \mathbb{E}\big [Z\big ] = d\log 2+\frac{\beta^2}{2}\sigma^2,
\]
where the variance $\sigma^2={\rm Var}(g_x)=(1-\alpha)G(x,x;\alpha)$ is given by \eqref{killed:00}.

\cite{P2018} studies free energy in Potts spin glass
and finds a general formula as $d\to \infty$ for the limit of
the free energy. In our context, the free energy is
\begin{align*}
  F_d &= \frac{1}{\beta}\mathbb{E} \big [\log Z]=
  \frac{1}{\beta}\mathbb{E}
  \left[\log\Big\{q^d+\beta\sum_{y\in{\cal V}_{q,d}} {\cal H}_{y}(\theta)+\frac{\beta^2}{2}\sum_{y\in{\cal V}_{q,d}}{\cal H}_y^2(\theta)+O(\beta^3)\Big\}\right]\\
  &=\frac{1}{\beta}
  \mathbb{E}\left[d\log q+\frac{\beta}{q^d} \sum_{y\in{\cal V}_{q,d}}{\cal H}_y(\theta)   +\frac{\beta^2}{2}
    \left\{\frac{1}{q^d}\sum_{y\in{\cal V}_{q,d}}{\cal H}_{y}^2(\theta)-\Big(\frac{1}{q^{d}}\sum_{y\in{\cal V}_{q,d}} {\cal H}_y(\theta)\Big)^2\right\}+O(\beta^3)\right]\\
    &=\frac{d}{\beta}\log q
     +\frac{\beta}{2q^d}
    \sum_{y\in{\cal V}_{q,d}}\mathbb{E}[{\cal H}_{y}^2(\theta)]-\frac{\beta}{2q^{2d}}\mathbb{E}\Big[\sum_{y\in{\cal V}_{q,d}} {\cal H}_y(\theta)\Big]^2+O(\beta^2).
\end{align*}
If $q=2$ and $b(y,x)=\delta_{xy}$, we have
\[
F_d=\frac{d}{\beta}\log 2+\frac{\beta}{2}\left\{\left(1-\frac{1}{2^d}\right)\sigma^2-\frac{1}{2^{2d}}\sum_{x\neq x^\prime}{\rm Cov}(g_x,g_{x\prime})\right\}+O(\beta^2).
\]
If $\beta$ is small the free energy can be minimized at positive $\beta$.
\begin{Example}
\begin{itemize}
\item[\phantom{a}]
\item[(a)]
${\cal H}_y(\theta)=-g_y$, when $b(y,x) = \delta_{xy}$, then
$
\mathbb{P}(y;\beta) = e^{\beta g_y}/Z.
$
\item[(b)]
It is possible to express
$
{\cal H}_y(\theta)=
\sum_{r\in {\cal V}_{q,d}}\theta_1^{y\dt r}\g_r
= \sum_{x\in {\cal V}_{q,d}}b(y,x)g_x,
$
then 
\[
\mathbb{P}(y;\beta)  = \frac{\exp \Big \{\beta\sum_{r\in {\cal V}_{q,d}}\theta_1^{y\dt r}\g_r\Big \}}{Z}.
\]
\item [(c)]
In a de Finetti model subsection \ref{deFinetti:gaussian}
\[
g_x = \sum_{r\in {\cal V}_{q,d}} \mathbb{E}\big [\prod_{j=1}^{d}Y_{\frac{1}{2}}[j]^{r[j]}\big ]\prod_{k=1}^{d}\theta_{x[k]}^{r[k]}\g_x
\]
An analogy with (\ref{Potts:00}) is that the spins are $(\theta_{x[j]})_{j=1}^{q-1}$
and the geometry is controlled by $r$ through $\mathbb{E}\big [\prod_{j=1}^{d}Y_{\frac{1}{2}}[j]^{r[j]}\big ]$.
\end{itemize}
\end{Example}
\section*{Acknowledgements} This paper was mainly written while the first author was visiting the second author at the Institute of Statistical Mathematics, Tachikawa, Tokyo in 2025. He thanks the Institute for their support and hospitality. 
The second author was supported in part by JSPS KAKENHI Grants 24K06876.
%


\begin{thebibliography}{99}
%
\bibitem[Berestycki and Powell(2025)]{B2016}
Berestycki, N. and Powell, E. (2025) Gaussian Free Field and Liouville Quantum Gravity. Cambridge: Cambridge University Press (Cambridge Studies in Advanced Mathematics), pp. i–vi.
%
\bibitem[Collevecchio and  Griffiths(2021)]{CG2021}
Collevecchio, A. and Griffiths, R. (2021) A class of random walks on the hypercube.  In and Out of Equilibrium 3: Celebrating Vladas Sidoravicius. Eulália Vares, M., Fernández, R., Renato Fontes, L. and Newman, C. M. (eds.). Cham Switzerland: Springer, Vol. 77. p. 265--298 34 p. (Progress in Probability; vol. 77). 
%
\bibitem[Collevecchio and  Griffiths(2025)]{CG2025}
Collevecchio, A., and Griffiths, R. (2025). Gaussian Free Fields on the Hypercube. J. Theor. Probab. 38 Article 35.
%
\bibitem[Diaconis and Griffiths(2014)]{DG2014} 
{Diaconis, P. and Griffiths R. C.} (2014)
An introduction to multivariate Krawtchouk polynomials and their applications.
{J. Stat. Plan. Inference} {154} 39--53.
%
\bibitem[Griffiths(1971)]{G1971}
{Griffiths, R. C.} (1971) 
Orthogonal polynomials on the multinomial distribution. {Austral. J. Statist.} {13} 27--35. Corrigenda (1972) { Austral. J. Statist.} { 14} 270.
%
\bibitem[Griffiths(2026)]{G2025}
{Griffiths, R. C.} (2026) 
Gaussian fields on a hypercube from long range random walks.  Stoch.  Process.  Their Appl.  197 Article 104939.
%
\bibitem[Griffiths and Mano(2026)]{GM2025a} 
Griffiths, R. C. and Mano, S. (2026) 
A class of Long Range Circulant Random walks on $Z_q^d$.  Stoch.  Process.  Their Appl.  201 Article 105038.
%
\bibitem[Mano(2025)]{M2025}
Mano, S. (2025). Gaussian free fields on Hamming graphs and
lattice spin systems. arXiv:2512.24199.
 %
 \bibitem[Mizukawa(2004)]{M2004}
Mizukara, H. (2004) Zonal spherical functions on the complex reflection groups and $(m+1,n+1)$-hypergeometric functions. {Advances in Math.} {184} 1-17.
%
\bibitem[Mizukawa and Tanaka(2004)]{MT2004}
Mizukawa, H., Tanaka, H., (2004) $(n+1,m+1)-$hypergeometric functions associated to character algebras. Proc. Amer. Math. Soc. 132 2613–2618.
%
\bibitem[Mizukawa(2011)]{M2011}
Mizukawa, H. (2011) Orthogonality relations for multivariate Krawtchouk polynomials. {SIGMA} {7} 017
%
\bibitem[Mizukawa(2018)]{M2018}
Mizukawa, H. (2018) 
Interactions between Ehrenfest’s urns arising from
group actions. Lett. Math. Phys. (2018) 108 1717–1728
%
  \bibitem[Werner and Powell(2021)]{WP2020}
 Werner, W., and Powell, E. (2021). Lecture notes on the Gaussian free field. Société Mathématique de France. \url{https://smf.emath.fr/publications/notes-de-cours-sur-le-champ-libre-gaussien}
 %
\bibitem[Panchenko(2018)]{P2018}
Panchenko, D. (2018) Free energy in the Potts spin glass.
Ann. Probab. 46, 829--864.
%
\bibitem[Potts(1952)]{P1952}
Potts, R. B. (1952). Some generalized order-disorder transformations. In Mathematical proceedings of the Cambridge philosophical society. 48 106--109. Cambridge University Press.
%
\bibitem[Wu(1982)]{W1982}
Wu, F. W. (1982) The Potts model.
Rev. Mod. Phys. 54 235--268.

\end{thebibliography}
\end{document}